\documentclass[11pt,letterpaper]{article}
\usepackage[margin=1.0in]{geometry}
\usepackage{amsmath,amsfonts,amssymb,amsrefs,
amsthm}
\usepackage{latexsym}
\usepackage[utf8]{inputenc}
\usepackage{tikz}
\usepackage{graphicx}
\usepackage{enumitem}
\usepackage{hyperref}
\usepackage[capitalise]{cleveref}
\usepackage{thmtools}
\usepackage{thm-restate}
\usepackage{caption}
\usepackage{xfrac}

\newtheorem{introtheorem}{Theorem}
\newtheorem{introprop}{Proposition}
\newtheorem{theorem}{Theorem}[section]
\newtheorem{corollary}[theorem]{Corollary}
\newtheorem{proposition}[theorem]{Proposition}
\newtheorem{lemma}[theorem]{Lemma}

\theoremstyle{remark}

\theoremstyle{definition}
\newtheorem{definition}[theorem]{Definition}

\newtheorem{example}[theorem]{Example}
\newtheorem{operation}[theorem]{Operation}

\title{Asymptotic and Assouad-Nagata dimension of finitely generated groups and their subgroups}
\author{Levi Sledd}

\newcommand{\llangle}{\langle \langle}
\newcommand{\rrangle}{\rangle \rangle}

\newcommand{\bd}{\partial}

\DeclareMathOperator{\Ker}{Ker}

\DeclareMathOperator{\Lab}{Lab}
\DeclareMathOperator{\asdim}{asdim}
\DeclareMathOperator{\asdimAN}{asdim_{\textnormal{AN}}}
\DeclareMathOperator{\dimAN}{dim_{\textnormal{AN}}}
\DeclareMathOperator{\diam}{diam}

\DeclareMathOperator{\PS}{PS}

\DeclareMathOperator{\supp}{supp}

\begin{document}
\maketitle

\begin{abstract}
We prove that for all $k,m,n \in \mathbb N \cup \{\infty\}$ with $4 \leq k \leq m \leq n$, there exists a finitely generated group $G$ with a finitely generated subgroup $H$ such that $\asdim(G) = k$, $\asdimAN(G) = m$, and $\asdimAN(H)=n$.  This simultaneously answers two open questions in asymptotic dimension theory.
\end{abstract}

\section{Introduction}
This is the second in a series of two papers on asymptotic dimension and Assouad-Nagata dimension of finitely generated groups: the first is \cite{Sledd}.  Asymptotic dimension ($\asdim$) and asymptotic Assouad-Nagata dimension ($\asdimAN$) are two distinct but related ways of defining the large-scale dimension of a metric space.  Each is invariant under quasi-isometry, and thus can be considered as an invariant of finitely generated groups.  For countable groups with proper left-invariant metrics, \emph{asymptotic} Assouad-Nagata dimension is equivalent to Assouad-Nagata dimension ($\dimAN$).  So from now on we use this shorter term when talking about groups, although we continue to denote it by $\asdimAN$.

Given a way of defining `dimension' for an algebraic structure, it is natural to ask whether it is monotonic with respect to substructures: that is, whether $A \leqslant B$ implies that the dimension of $A$ is no greater than the dimension of $B$.  Is our dimension like that of a vector space, where this natural monotonicity holds, or is it like the rank of a free group, where it fails spectacularly?  Since $\asdim$ is actually a \textit{coarse} invariant, it follows that $\asdim$ is well defined for all countable groups, and if $G$ is a countable group and $H \leqslant G$, then $\asdim(H) \leq \asdim(G)$. In this paper we show that Assouad-Nagata dimension behaves quite differently.  Namely, we prove the following theorem.

\begin{introtheorem}\label{IntroThm1}
For any $k,m,n \in \mathbb N \cup \{\infty\}$ with $4 \leq k \leq m \leq n$, there exist finitely generated, recursively presented groups $G$ and $H$ with $H \leqslant G$, such that
\begin{align*}
\asdim(G) &= k\\
\asdimAN(G) &= m\\
\asdimAN(H) &= n \, .
\end{align*}
\end{introtheorem}

In \cite{Higes2}, Higes constructs an infinitely generated, locally finite abelian group, and a proper left-invariant metric with respect to which the group has asymptotic dimension 0 but infinite Assouad-Nagata dimension.  In \cite{Brodskiy_Dydak_Lang}, Brodskiy, Dydak, and Lang construct finitely generated groups with a similar gap, showing that (for example) $\mathbb Z_2 \wr \mathbb Z^2$ has asymptotic dimension 2 but infinite Assouad-Nagata dimension.  Previously it was not known whether a finitely generated group $G$ could satisfy $\asdim(G) < \asdim_{AN}(G) < \infty$ (Question (2) of \cite{Higes2}), nor was it known whether a finitely generated group could contain a finitely generated subgroup of greater Assouad-Nagata dimension (Questions 8.6 and 8.7 of \cite{Brodskiy_etal}).  With \cref{IntroThm1}, we show that both these things are possible.

If $H \leqslant G$ but $\asdimAN(H)  > \asdimAN(G)$, it must be that $H$ is distorted in $G$, and that this distortion collapses $H$ to a space of lesser Assouad-Nagata dimension in $G$.  However, distortion does not always affect the Assouad-Nagata dimension of the distorted subgroup.  For example, in $BS(1,2) =\langle a, b \mid b^{-1}aba^{-2} \rangle$, the subgroup $\langle a \rangle$ is distorted, but still has Assouad-Nagata dimension $1$.  We call distortion which affects Assouad-Nagata dimension \emph{Assouad-Nagata dimension distortion}.  The author hopes that Assouad-Nagata dimension distortion will be an interesting phenomenon to study in its own right, and that more examples can be found in nature.

The paper is organized as follows.  In Section 1, we fix countable group $K$, constructed as a direct sum of cyclic groups of increasing order.  We then show that for each $m, n \in \mathbb N \cup \{\infty\}$ with $m<n$, there are two different proper left-invariant metrics on $K$ such that $\asdimAN(K) = m$ with respect to one, and $\asdimAN(K) = n$ with respect to the other.  

In Section 2, we use techniques from small cancellation theory to establish a highly technical lemma.  This lemma allows us to quasi-isometrically embed $K$, with respect to each proper left-invariant metric, into a finitely generated group.  

In Section 3, we embed $K$ into finitely generated groups $A$ and $B$.  This is done is such a way that, calling $K_A$ the copy of $K$ in $A$ and $K_B$ the copy of $K$ in $B$, we have that $\asdimAN(K_A) = m$ and $\asdimAN(K_B) = n$.  We then identify the two with an isomorphism $\phi : K_A \to K_B$, and let $G = A *_\phi B$.  Our technical small cancellation lemma comes back to help us a second time by showing that $\phi$ `crushes' the image of $K_B$ in $G$ to the size of $K_A$.  With a few calculations using well-known extension theorems for Assouad-Nagata dimension, we are able to prove the following.
\begin{introprop}\label{introProp}
For any $m, n \in \mathbb N \cup \{\infty\}$ with $m < n$, there exists a group $G = A *_\phi B$ where $G$, $A$, and $B$ are finitely generated and recursively presented, such that
\begin{align*}
1 &\leq \asdim(G) \leq 2\\
m+1 &\leq \asdimAN(G) \leq m+2\\
n+1 &\leq \asdimAN(B) \leq n+2 \, .
\end{align*}
\end{introprop}
Using the free product formulas for asymptotic and Assouad-Nagata dimension and the Morita theorem for Assouad-Nagata dimension, it is then easy to derive \cref{IntroThm1} from \cref{introProp}.  

There are many technical restrictions placed on the presentations of $A$ and $B$ from \cref{introProp}.  In Section 4 we give explicit presentations where these conditions are satisfied.  Curiously, although we are able to give an explicit presentation of a group $G$ satisfying the conditions of \cref{introProp}, we are not quite able to do the same for \cref{IntroThm1}.  However, we can explicitly give presentations of two groups, one of which must be a group satisfying the conclusion of \cref{IntroThm1}.

\section{Adapting a construction of Higes}

We refer the reader to \cite{Sledd}, Section 1 for basic conventions and notation regarding metric spaces, as well as definitions of the terms asymptotic dimension ($\asdim$), asymptotic Assouad-Nagata dimension ($\asdimAN$), and control function.  We assume that the reader is familiar with the notions of quasi-isometry and bi-Lipschitz equivalence.  Occasionally we will mention terms such as `coarse' map/embedding/equivalence.  Since we will not need to work with these directly, we do not give a definition here, but one may be found in any text on coarse geometry, for example \cite[pp. 9]{Roe}.  What matters to us is that, if $X$ and $Y$ are metric spaces which are coarsely equivalent, then $\asdim(X) = \asdim(Y)$, and that a quasi-isometry or bi-Lipschitz map is a special case of a coarse equivalence.

In this paper, we adopt the convention that the Cartesian product of two metric spaces is always endowed with the $\ell^1$ product metric.  That is, if $X$ and $Y$ are metric spaces, then $X \times Y$ is equipped with the metric defined by
$$d_{X \times Y}((x,y),(x',y')) = d_X(x,x')+d_Y(y,y')$$
for all $x,x' \in X$ and $y,y' \in Y$.  With this convention in mind, if $\sim$ stands for either ``is coarsely equivalent to," ``is quasi-isometric to," or ``is bi-Lipschitz equivalent to," then we have that $X \sim X'$ and $Y \sim Y'$ implies $X \times Y \sim X' \times Y'$.  In addition, $\asdim$ and $\asdimAN$ are subadditive with respect to taking direct products, in a sense that is made precise by the following two theorems.  We will use them often throughout this paper.

\begin{lemma}\cite{Bell_Dranishnikov, Brodskiy_etal}\label{subadditivity}
Let $X, Y$ be metric spaces.  Then
\begin{align*}
\asdim(X \times Y) &\leq \asdim(X) + \asdim(Y)\\
\asdimAN(X \times Y) &\leq \asdimAN(X) + \asdimAN(Y).
\end{align*}
\end{lemma}

\subsection{Normed groups}

We denote the identity element of an arbitrary group by 1, and of an abelian group by 0.  Let $G$ be a group.  A \emph{norm} on $G$ is a function $\|\cdot\|: G \to \mathbb R_0^+$ such that, for all $g, h \in G$,
\begin{itemize}
\item $\|g\| = 0$ if and only if $g=1$.
\item $\|g\| = \|g^{-1}\|$.
\item $\|gh\| \leq \|g\|+\|h\|$.
\end{itemize}
Some authors call this a length function or weight function on $G$.

A norm is \emph{proper} if $\{g \in G \mid \|g\| \leq N \}$ is finite for all $N \geq 0$.  There is a natural one-to-one correspondence between norms and left-invariant metrics, given by $d(g,h) = \|g^{-1}h\|$ and $\|g\| = d(1,g)$, and a left-invariant metric on a group is proper if and only if the corresponding norm is proper.  Every countable group admits a proper norm, and any two proper norms on the same countable group are coarsely equivalent \cite[Proposition 1.1]{Dranishnikov_Smith2}.  Thus $\asdim$ is an invariant of countable groups: in particular, if $G$ is a countable group and $H \leqslant G$, then $\asdim(H) \leq \asdim(G)$.  It is easy to show that a countable group has asymptotic dimension zero if and only if it is locally finite, a fact which we will use many times.

Formally, a normed group should be an ordered pair $(G, \|\cdot\|_G)$.  But from now on, whenever we say that $G$ is a \emph{normed} group, it is understood that $G$ is equipped with a norm, which is always called $\|\cdot\|_G$.  With this convention in mind we eliminate the norm from the notation wherever possible.

If $G$ is a normed group and $s$ is a positive real number, then the function $s\|\cdot\|_G: G \to \mathbb R_0^+, g \mapsto s\|g\|_G$ is also a norm on $G$.  We call the normed group $(G, s\|\cdot\|_G)$ a \emph{scaled normed group}, and denote it briefly by $sG$.

Given two normed groups $G_0$ and $G_1$ and scaling constants $s_0, s_1$, we define their \emph{scaled direct product} $s_0G_0 \times s_1G_1$ to be the group $G_0 \times G_1$ endowed with the norm $\|\cdot\|_{(s_0, s_1)}$ defined by
$$\|(g_0,g_1)\|_{(s_0,s_1)} = s_0\|g_0\|_{G_0} + s_1\|g_1\|_{G_1}$$
for all $g_0 \in G_0$ and $g_1 \in G_1$.  This is just the $\ell^1$ product norm on $s_0G_0 \times s_1G_1$.  For any $k \in \mathbb N$, we define the scaled direct product of finitely many scaled normed groups $\prod_{i=0}^k s_iG_i$ by iterating this construction.  Note that for finite direct products we have that $\prod_{i=0}^k s_iG_i$ is bi-Lipschitz equivalent to $\prod_{i=0}^k G_i$ without scaling.

To avoid frequently having to state that certain sets are nonempty, we declare $\prod_{i \in \emptyset} G_i$ to be the trivial group.  Let $I$ be a set and $(G_i)_{i \in I}$ an $I$-tuple of groups.  For $g = (g_i)_{i \in I} \in \prod_{i \in I} G_i$, we denote the \emph{support} of $g$ by $\supp(g)$; that is, $\supp(g) = \{i \in I \mid g_i \neq 1\}$.  By definition $\bigoplus_{i \in I} G_i$ is the subgroup of $\prod_{i \in I} G_i$ consisting of all $g \in \prod_{i \in I} G_i$ such that $\supp(g)$ is finite.  The notion of scaled direct products can then be extended to general direct sums in the following natural way. 

\begin{definition}
Let $I$ be a set, let $(G_i)_{i \in I}$ be an $I$-tuple of normed groups, and let $s = (s_i)_{i \in I}$ an $I$-tuple of scaling constants.  Let $G = \bigoplus_{i \in I} G_i$.  Then the \emph{scaled direct sum} $\bigoplus_{i \in I} s_iG_i$ is defined to be the normed group $(G, \|\cdot\|_s)$, where $\|\cdot\|_s$ is given by
$$\|g\|_s = \sum_{i \in I} s_i\|g_i\|_{G_i}$$
for all $g \in G$.  We call $\|\cdot\|_s$ the norm \emph{induced} by $s$.
\end{definition}

\begin{lemma}\label{WLOGIntegers}
Let $I$ be a set, $s=(s_i)_{i \in I}$ an $I$-tuple of scaling constants bounded away from zero.  Then $\bigoplus_{i \in I} s_iG_i$ is bi-Lipschitz equivalent to $\bigoplus_{i \in I} s_i'G_i$, where $s_i'$ is a positive integer for all $i \in I$.
\end{lemma}
\begin{proof}
Suppose that $\varepsilon > 0$ is such that $s_i \geq \varepsilon$ for all $i \in I$.  Let $s' = (s_i')_{i \in I} = (\lceil s_i \rceil)_{i \in I}$, and let $g = (g_i)_{i \in I} \in \bigoplus_{i \in I} G_i$.  Then clearly $\|g\|_s \leq \|g\|_{s'}$, and
$$\|g\|_{s'} = \sum_{i \in I} \lceil s_i \rceil \|g_i\|_{G_i} \leq \sum_{i \in I} \left( \tfrac{s_i+1}{s_i} \right) s_i\|g_i\|_{G_i} \leq \left(1+\tfrac{1}{\varepsilon}\right) \sum_{i \in I} s_i\|g_i\|_{G_i} = \left(1+\tfrac{1}{\varepsilon}\right)\|g\|_s \, .$$
\end{proof}

\subsection{A fixed group with varying norms}

The next set of lemmas deal specifically with direct sums of cyclic groups.  Here we assume that a cyclic group comes equipped with the natural norm, that is $\|x\|_{\mathbb Z_\ell} = \min(x, \ell-x)$ for all $x \in \mathbb Z_\ell$, and $\|x\|_{\mathbb Z} = |x|$ for all $x \in \mathbb Z$.  Unless otherwise noted, tuples are sequences indexed by $\mathbb N$, e.g. $(s_i)$ stands for $(s_i)_{i \in \mathbb N}$.

\begin{definition}\label{geodForm}
Let $(x_i) \in \bigoplus_{i \in \mathbb N} \mathbb Z_{\ell_i}$.  The \emph{geodesic form} of $(x_i)$ is the unique sequence of integers $(y_i)$ such that for all $i \in \mathbb N$,
\begin{itemize}
\item $y_i \equiv x_i \mod \ell_i \, \text{, and}$
\item $y_i \in \left\{-\left\lfloor \tfrac{\ell_i-1}{2}\right\rfloor, \ldots, -1,0,1, \ldots, \left\lfloor \tfrac{\ell_i}{2} \right\rfloor\right\}_{\, .}$
\end{itemize}
\end{definition}

Note that if $s=(s_i)$ is a sequence of scaling constants, $x = (x_i) \in \bigoplus_{i \in \mathbb N} s_i\mathbb Z_{\ell_i}$, and $(y_i)$ is the geodesic form of $x$, then we have
\begin{equation}\label{geodFormNorm}
\|x\|_s = \sum_{i \in \mathbb N} s_i|y_i| \, .
\end{equation}

\begin{definition}\label{cubeDef}
For $s \in \mathbb R^+$ and $k, n \in \mathbb N$, assume that the set $$s\{0, \ldots, k\}^n = \{0, s, 2s, \ldots, ks\}^n \subset \mathbb R^n$$
is equipped with the $\ell^1$ metric.  Then an \emph{expanded $n$-dimensional cube} is a space isometric to $s\{0, \ldots, k\}^n$ for some $s \geq 1$ and $k \in \mathbb N$.
\end{definition}

In accordance with \cref{cubeDef}, whenever $s$ is a scaling constant and $s \geq 1$, we call $s$ an \emph{expansion constant}.  Sequences of expanded cubes are useful for establishing lower bounds on the asymptotic Assouad-Nagata dimension of a metric space. 

\begin{lemma}\cite[Corollary 2.7]{Higes}\label{cubeLemma1}
Let $X$ be a metric space, $n \in \mathbb N$.  If $X$ contains a sequence of expanded $n$-dimensional cubes $s_j\{0, \ldots, k_j\}^n$ where $\lim_{j \to \infty} k_j = \infty$, then $\asdimAN(X) \geq n$.
\end{lemma}

Suppose that $P$ is a set with $|P| \geq n$, $(\ell_i)_{i \in P}$ is a $P$-tuple of natural numbers, and $s_P$ is an expansion constant.  Let $k_P$ be a natural number with $k_P \leq \min\{\ell_i/2 \mid i \in P\}$.  Then by (\ref{geodFormNorm}), $s_P \bigoplus_{i \in P} \mathbb Z_{\ell_i}$ contains an expanded $n$-dimensional cube $s_P\{0, \ldots, k_P\}^n$.  With this observation and \cref{cubeLemma1}, one can construct a group which can achieve any positive Assouad-Nagata dimension.  The idea is to take a direct sum of cyclic groups, block every $n$ of them together, and scale the blocks appropriately.  In \cite{Higes} Higes uses this idea to construct, for any $n \in \mathbb Z^+ \cup \{\infty\}$, a normed group $G_n$ with asymptotic dimension zero but Assouad-Nagata dimension $n$.  However, in Higes' examples, if $m \neq n$, then $G_m$ and $G_n$ are not isomorphic.  For our purposes, it is important that the group be fixed, with only the norm varying.  The rest of this section is devoted to working out the details of this construction.  To smooth the process, we introduce the following ad hoc notation.

\begin{definition}\label{partitions}
For each $m \in \mathbb Z^+ \cup \{\infty\}$, let $\mathcal P_m = \{P_{(m,j)} \mid j \in \mathbb N\}$ be the partition of $\mathbb N$ given by
\[
P_{(m,j)} = 
\begin{cases}
\{jm, jm+1, \ldots, (j+1)m-1\} & \text{ if } m \in \mathbb Z^+\\
\{j^2, j^2+1, \ldots, (j+1)^2-1\} & \text{ if } m=\infty \, .
\end{cases}
\]
\end{definition}

\begin{definition}
Let $s=(s_i)$ be a sequence, $m \in \mathbb Z^+ \cup \{\infty\}$.  Let the \emph{$m$-inflation} of $s$, denoted $m \times s$, be the sequence defined by
$$(m \times s)_i = s_j \Leftrightarrow i \in P_{(m,j)}\, .$$
\end{definition}

For example, if $s = (1,2,3, \ldots)$, then $2 \times s = (1,1,2,2,3,3, \ldots)$ and $\infty \times s = (1, 2, 2, 2, 3, 3, 3, 3, \ldots)$.  By definition,
\begin{equation}
s_i =
\begin{cases}
(m \times s)_{im} & \text{ if } m \in \mathbb Z^+ \\
(m \times s)_{i^2} &\text{ if } m = \infty \,
\end{cases} 
\quad \text{and} \quad 
(m \times s)_i = 
\begin{cases}
s_{\lfloor i/m \rfloor} & \text{ if } m \in \mathbb Z^+ \\
s_{\lfloor \sqrt i \rfloor} &\text{ if } m = \infty \, .
\end{cases}
\end{equation}

\begin{lemma}\label{ANLowerBound}
Let $d \in \mathbb N$, and let $(c_0, \ldots, c_{d-1})$ be a finite sequence of scaling constants.  Let $m \in \mathbb Z^+ \cup \{\infty\}$ be fixed, let $(s_i)$ be an increasing sequence of expansion constants, and let $(\ell_i)$ be an increasing sequence of positive integers.  Let
\begin{align*}
Z_d &= \bigoplus_{i=0}^{d-1} c_i\mathbb Z \, , &&
K_m = \bigoplus_{i \in \mathbb N} (m \times s)_i \mathbb Z_{\ell_i} \, .
\end{align*}
Then $\asdimAN(Z_d \times K_m) \geq d+m$.
\end{lemma}
\begin{proof}
By \cref{WLOGIntegers}, we may assume without loss of generality that all $s_i$ are positive integers.  Since finite direct products preserve bi-Lipschitz equivalence, we may also assume that all $c_i$ are equal to 1, so that $Z_d = \mathbb Z^d$.

Now note that
$$Z_d \times K_m = \mathbb Z^d \times \bigoplus_{j \in \mathbb N} \left( s_j \bigoplus_{i \in P_{(m,j)}} \mathbb Z_{\ell_i} \right)_{\, ,}$$
where $\mathbb Z^d \times s_j \bigoplus_{i \in P_{(m,j)}} \mathbb Z_{\ell_i}$ is an isometrically embedded subgroup for each $j \in \mathbb N$.  Let 
$$k_j = \min\{\lfloor \ell_i/2 \rfloor \mid i \in P_{(m,j)}\} = 
\begin{cases}
\lfloor \ell_{jm}/2 \rfloor & \text{ if } m \in \mathbb Z^+ \\
\lfloor \ell_{j^2}/2 \rfloor & \text{ if } m = \infty \, .
\end{cases}$$
Then $\lim_{j \to \infty} k_j = \infty$.

If $m \in \mathbb Z^+$, then $|P_{(m,j)}| = m$ for all $j \in \mathbb N$.  Then since $s_j$ is an integer, $\mathbb Z^d \times s_j \bigoplus_{i \in P_{(m,j)}} \mathbb Z_{\ell_i}$ contains the expanded $(d+m)$-dimensional cube $s_j\{0, \ldots, k_j\}^{d+m}$ for all $j \in \mathbb N$.  Since $\lim_{j \to \infty} k_j = \infty$, by \cref{cubeLemma1} we have $\asdimAN(\mathbb Z^d \times K_m) \geq d+m$.

If $m = \infty$, let $n \in \mathbb Z^+$.  Then  $|P_{(m,j)}| = (j+1)^2-j^2 = 2j+1 \geq n$ for all $j \geq n$.  Therefore $s_j \bigoplus_{i \in P_{(m,j)}} \mathbb Z_{\ell_i}$ contains the expanded $n$-dimensional cube $s_j\{0, \ldots, k_j\}^n$ for all $j \geq n$.  Since $\lim_{j \to \infty} k_j = \infty$, by \cref{cubeLemma1} we have $\asdimAN(K_\infty) \geq n$.  Since $n \in \mathbb Z^+$ was chosen arbitrarily, $\asdimAN(K_\infty) = \infty$, thus $\asdimAN(Z_d \times K_\infty) = \infty$.
\end{proof}

Now, in the notation of \cref{ANLowerBound}, we wish to impose certain conditions on the sequence $(s_i)$ of expansion constants to guarantee that $\asdimAN(Z_d \times K_m) = d+m$ exactly.   We will use a lemma of Higes; in order to do so we need to introduce a little notation, and consider a different norm on countable direct sums of scaled normed groups.  

\begin{definition}
Let $(G_i)$ be a sequence of normed groups and $s=(s_i)$ a sequence of scaling constants.  Let $G = \bigoplus_{i \in \mathbb N} G_i$.  For convenience, let us define the \emph{height} function $h: G \to \mathbb N$ by
\[
h(g)=
\begin{cases}
0 & \text{if }g = 1 \\ 
\max(\supp(g)) &\text{ otherwise.}
\end{cases}
\]
Now define the \emph{quasi-ultranorm} on $G$ induced by $s$, denoted $\|\cdot\|_s^{\textnormal{qu}}$, by 
\begin{equation}\label{quNorm}
\|g\|_s^{\textnormal{qu}} = s_h\|g_h\|_{G_h}
\end{equation}
for all $g = (g_i) \in G$, where $h = h(g)$.
\end{definition}

In \cite{Higes}, Higes calls the metric associated to this norm the \emph{quasi-ultrametric} generated by the sequence of metrics $(d_{G_i})$, where $d_{G_i}$ is the metric associated to the scaled norm $s_i\|\cdot\|_{G_i}$ for each $i \in \mathbb N$.  For this reason we call the norm in (\ref{quNorm}) the quasi-ultranorm on $G$ induced by $s$, and put `$\textnormal{qu}$' in the superscript.  The next lemma says that if all $G_i$ are finite then, under mild assumptions about the growth of the sequence $s$, the norms $\|\cdot\|_s$ and $\|\cdot\|_s^{\textnormal{qu}}$ are, for our purposes, interchangeable.

\begin{lemma}\label{quNormEquivalence}
Let $(G_i)$ be a sequence of normed groups and $s = (s_i)$ a sequence scaling constants.  Let $G = \bigoplus_{i \in \mathbb N} G_i$.  Suppose that $G_i, \|\cdot\|_{G_i}, s_i$ satisfy the following conditions for all $i \in \mathbb N$:
\begin{itemize}
\item $\|g_i\|_{G_i} \geq 1$ for all $g_i \in G_i \smallsetminus \{1\}$.
\item $\diam(G_{i+1}) \geq \diam(G_i)$.
\item $s_{i+1} \geq 2s_i\diam(G_i)$.
\end{itemize}
Then the norm $\|\cdot\|_s$ and quasi-ultranorm $\|\cdot\|_s^{\textnormal{qu}}$ induced by $s$ are bi-Lipschitz equivalent.
\end{lemma}

\begin{proof}
Clearly $\|g\|_s^{\textnormal{qu}} \leq \|g\|_s$ for all $g \in G$.

We now prove by induction on $h(g)$ that $\|g\|_s \leq 2\|g\|_s^{\textnormal{qu}}$.  This is clear when $h(g)=0$.  Now suppose that $h(g) = k \geq 1$.  Write $g$ as $g'g''$, where $g_j' = g_j$ exactly when $j=k$ and is equal to 1 otherwise, and $h(g'') = i < k$.  Then we have
\begin{align*}
\|g\|_s &\leq \|g'\|_s + \|g''\|_s = \|g'\|_s^{\textnormal{qu}} + \|g''\|_s \leq \|g'\|_s^{\textnormal{qu}} + 2\|g''\|_s^{\textnormal{qu}}\\
&\leq \|g'\|_s^{\textnormal{qu}} + 2s_i\diam(G_i) \leq \|g'\|_s^{\text{qu}}+2s_{k-1}\diam(G_{k-1})\\ 
&\leq \|g'\|_s^{\textnormal{qu}} + s_k \leq 2\|g'\|_s^{\textnormal{qu}} = 2\|g\|_s^{\textnormal{qu}} \, .
\end{align*}
\end{proof}

\begin{lemma}\cite[Proof of Corollary 4.11]{Higes}\label{ANUpperBound}
Let $(\ell_i)$ be an increasing sequence of positive integers with $\ell_0 \geq 2$.  Let $m$ be a fixed positive integer.  Let $s=(s_i)$ be a sequence of expansion constants such that
$$s_{i+1} \geq 1+s_i\diam(\mathbb Z_{\ell_i}^m) =  1+(m\lfloor \ell_i/2 \rfloor)s_i \, .$$
Let $K_m^{\textnormal{qu}} = (\bigoplus_{i \in \mathbb N} \mathbb Z_{\ell_i}^m, \|\cdot\|_s^{\textnormal{qu}})$.  Then for any $k \in \mathbb N$ we have $\asdimAN(\mathbb Z^k \times K_m^{\textnormal{qu}}) = k+m$.
\end{lemma}

We use this lemma in the case $k=0, m=1$ to obtain the slightly generalized lemma that we need.

\begin{lemma}\label{K_nConstruction}
Let $d \in \mathbb N$, and let $(c_0, \ldots, c_{d-1})$ be a finite sequence of scaling constants.  Let $(\ell_i)$ be a sequence of positive integers, and let $m \in \mathbb Z^+ \cup \{\infty\}$ be fixed.  Let $(s_j)$ be an increasing sequence of expansion constants such that, if $m \in \mathbb Z^+$, we have
$$s_{j+1} \geq (\ell_{(j+1)m}) s_j$$
for all $j \in \mathbb N$.  Now let 
\begin{align*}
Z_d = \bigoplus_{i=0}^{d-1} c_i \mathbb Z &&
K_m = \bigoplus_{i \in \mathbb N} (m \times s)_i \mathbb Z_{\ell_i} \, .
\end{align*}
Then $\asdimAN(Z_d \times K_m) = d+m$.
\end{lemma}
\begin{proof}
The lower bound is established in \cref{ANLowerBound}.  For the upper bound, suppose that $m \in \mathbb Z^+$.  Then
$$K_m = \bigoplus_{r=0}^{m-1} \left(\bigoplus_{j \in \mathbb N} s_j \mathbb Z_{\ell_{jm+r}} \right)_{\, .}$$
Since $(\ell_i)$ is increasing, for all $j \in \mathbb N$ and $r \in \{0, \ldots, m-1\}$ we have that
$$s_{j+1} \geq (\ell_{(j+1)m})s_j \geq (\ell_{jm+r})s_j \geq (2 \lfloor \ell_{jm+r}/2 \rfloor)s_j = (2\diam(\mathbb Z_{\ell_{jm+r}}))s_j \geq 1+s_j\diam(\mathbb Z_{\ell_{jm+r}}) \, .$$
Therefore for any fixed $r \in \{0, \ldots, m-1\}$, the sequences $(\ell_{jm+r}), (\mathbb Z_{\ell_{jm+r}}),$ and $(s_j)$ together satisfy the assumptions of Lemmas \ref{quNormEquivalence} and \ref{ANUpperBound}.  Hence for all $r \in \{0, \ldots, m-1\}$,
$$\asdimAN\left(\bigoplus_{j \in \mathbb N} s_j \mathbb Z_{jm+r} \right) = \asdimAN\left(\bigoplus_{j \in \mathbb N} \mathbb Z_{jm+r}, \|\cdot\|_s \right) = \asdimAN\left(\bigoplus_{j \in \mathbb N} \mathbb Z_{jm+r}, \|\cdot\|_s^{\textnormal{qu}} \right) = 1 \, .$$
Thus by \cref{subadditivity},
$$\asdimAN(K_m) \leq \sum_{r=0}^{m-1}\asdimAN\left(\bigoplus_{j \in \mathbb N} s_j \mathbb Z_{jm+r} \right) \leq m \, ,$$
and $\asdimAN(Z_d)= \asdimAN(\mathbb Z^d) =d$.  Therefore $\asdimAN(Z_d \times K_m) \leq d+m$.
\end{proof}

The importance of \cref{K_nConstruction} lies in the fact that if $(\ell_i)$ is fixed and $m,n \in \mathbb Z^+ \cup \{\infty\}$ are distinct, then $K_m$ and $K_n$ are merely the same group with different norms.  Later, we will construct two finitely generated groups $A$ and $B$ with subgroups that are isomorphic and bi-Lipschitz equivalent to $K_m$ and $K_n$, respectively.  Since $K_m$ and $K_n$ are isomorphic, we construct a finitely generated group $G$ which is the amalgamated product of $A$ and $B$ along an isomorphism between $K_m$ and $K_n$.  The isomorphism `collapses' $K_n$, so that the Assouad-Nagata dimension of $G$ is not much more than $m$, while the Assouad-Nagata dimension of $B$ is at least $n$.  To construct $A, B$, and $G$ such that all of the aforementioned geometric properties hold, we use some small cancellation theory.  This is the topic of the next section.

\section{van Kampen diagrams and the $C'(\lambda)$ condition}

The goal of this section is to prove \cref{quasiGeod}, which states that words of a certain form are quasigeodesic in certain central extensions of $C'(\lambda)$ groups, where $0 < \lambda < \sfrac{1}{12}$.  This is a generalization \cite[Lemma 5.10]{Olshanskii_Osin_Sapir}, originally used to construct finitely generated groups with circle-tree asymptotic cones.  The proof of \cref{quasiGeod} is a technical argument that involves performing surgery on van Kampen diagrams.

We assume that the reader is familiar with the $C'(\lambda)$ condition and the notion of a van Kampen diagram.  However, there are myriad definitions of van Kampen diagram in the literature, and for our purposes it is necessary to define the $C'(\lambda)$ condition in a way which, though clearly equivalent to the usual definition, is slightly non-standard.  Therefore in Sections \ref{C'lambdaCondition} and \ref{vanKampenDiagrams} we fix terminology and notation, and provide all necessary definitions for the following sections.  

In \cref{vanKampenOps}, we define signed and unsigned $r$-face counts, where $r$ is a relation of a presentation.  We also introduce various operations on van Kampen diagrams, and examine how each of these operations affects the signed and unsigned $r$-face counts.  Our approach is to treat van Kampen diagrams as graphs embedded in the plane, so that the 2-cells are simply the bounded faces enclosed by the graph.  In this way we manipulate van Kampen diagrams directly in the plane and keep topological considerations to a minimum.  In \cref{vKDsOverC'1-6}, we collect some facts about van Kampen diagrams over $C'(\sfrac{1}{6})$ presentations that are used in the proof of \cref{quasiGeod}.  Finally, in \cref{technicalLemma}, we prove \cref{quasiGeod}.  

\subsection{The $C'(\lambda)$ condition}
\label{C'lambdaCondition}

Let $S$ be a set.  Let $S^{-1}$ be the set of formal inverses of $S$, let $1$ be a new symbol not in $S$, and declare $1^{-1}=1$.  Let
\begin{equation}
\begin{split}
S_1 &= S \cup \{1\}\\
S_\circ &= S \cup S^{-1} \cup \{1\}.
\end{split}
\end{equation}

The length of a word $w$ in the free monoid $S_\circ^*$ is denoted $|w|$.  There is a unique word of length 0 called the \emph{empty word} and denoted $\varepsilon$.  We define $w^0$ to be $\varepsilon$ for any $w \in S_\circ^*$.  A word $w \in S_\circ^*$ is \emph{reduced} if $w$ does not contain a subword of the form $1, ss^{-1},$ or $s^{-1}s$ for any $s \in S$, and \emph{cyclically reduced} if every cylcic shift of $w$ (including $w$ itself) is reduced.  

Let $R$ be a language over the alphabet $S_\circ$, that is, $R \subseteq S_\circ^*$.  Then  $R_*$ denotes the closure of $R$ under taking cyclic shifts and formal inverses of its elements.  We say that $R$ is \emph{reduced} if every element of $R$ is reduced, and \emph{cyclically reduced} if $R_*$ is reduced.  We say that $R$ is \emph{cyclically minimal} if it does not contain two distinct words, one of which is a cyclic shift of the other word or its inverse.  That is, $R$ is cyclically minimal if $R \cap \{r\}_* = \{r\}$ for each $r \in R$. 

A \emph{presentation} is a pair $\langle S \mid R \rangle$, where $S$ is a set and $R \subseteq S_\circ^*$.  The notation $G=\langle S \mid R \rangle$ means that $\langle S \mid R \rangle$ is a presentation and $G \cong F(S)/\llangle R \rrangle$, where $F(S)$ is the free group with basis $S$, and $\llangle R \rrangle$ is the normal closure of $R$ as a subset of $F(S)$.  

If $G$ is a group generated by a set $S$, there is a natural monoid epimorphism from $S_\circ^*$ to $G$ that evaluates a word in $S_\circ^*$ as a product of generators and their inverses, and sends $1$ to the identity element.  If both $G$ and $S$ are understood, then for a word $w \in S_\circ^*$ we denote the image of $w$ under this homomorphism by $\bar w$.  For a group element $g \in G$, we write $w=_G g$ to abbreviate that $\bar w = g$.  The \emph{word norm} on $G$ with respect to $S$ is defined by
$$\|g\|_G = \min\{|w| \mid w \in S_\circ^*, w =_G g\} \, .$$
We omit the generating set from the notation since any other choice of finite generating set yields a norm which is bi-Lipschitz equivalent. What matters is that the generating set is fixed throughout any proof in which the word norm plays a role.

A word $w \in S_\circ^*$ is called \emph{geodesic} in $G$ if $|w| = \|\bar w\|_G$.  If $u,w \in S_\circ^*$, $g \in G$, $w$ is geodesic, and $w =_G u =_G g$, then $w$ is called a \emph{geodesic representative} of $u$ or of $g$ in $G$.  If $K,C \geq 0$ are fixed constants, then we say that a word $w \in S_\circ^*$ is \emph{$(K,C)$-quasigeodesic} in $G$ if $|w| \leq K\|\bar w\|_G+C$.

Given two words $u, v \in S_\circ^*$, we say that $p$ is a \emph{piece} (of $u$ and of $v$) if there exists $u' \in \{u\}_*, v' \in \{v\}_*$ such that $p$ is a common prefix of $u'$ and $v'$.

\begin{definition}
Let $S$ be a set, $R \subseteq S_\circ^*$ a language, and $\lambda$ a real number with $0 < \lambda < 1$.  Then $R$ satisfies $C'(\lambda)$ if, whenever $u,v \in R$ and $u' \in \{u\}_*, v' \in \{v\}_*$ witness that $p$ is a piece of $u$ and $v$, then either $u'=v'$ or $|p| < \lambda \min(|u|, |v|)$.
\end{definition}
In this case we say that $R$ is a $C'(\lambda)$ language.  If $G$ is a group and $G=\langle S \mid R \rangle$ for some $C'(\lambda)$ language $R$, then $\langle S \mid R \rangle$ is called a $C'(\lambda)$ presentation and $G$ is called a $C'(\lambda)$ group.

In most treatments of the $C'(\lambda)$ condition, it is assumed that $R=R_*$, and a piece is defined to be a common prefix of two distinct words in $R$.  In our case, however, it is important to assume that $R$ is cyclically minimal (in particular $R \neq R_*$), in order to ensure that the signed $r$-face count (\cref{signedrFaceCount} below) is well defined.  For this reason we give the definition above, which, though not the usual definition of the $C'(\lambda)$ condition, is clearly equivalent.

\subsection{van Kampen Diagrams}\label{vanKampenDiagrams}

Let $\Gamma$ be a connected graph.  By a \emph{path} in $\Gamma$ we mean a combinatorial path, which may have repeated edges or vertices: in graph-theoretic terms, our `path' is really a walk.  Since points in the interiors of edges generally don't matter to us, we write $x \in \Gamma$ to mean that $x \in V(\Gamma)$.  Likewise, if $\alpha$ is a path in $\Gamma$, then $x \in \alpha$ means that $x$ is a vertex visited by $\alpha$.  

Let $\Gamma$ be any directed graph, and suppose that $\Lab : E(\Gamma) \to S_1$ (see (1) above) is a function which assigns labels from $S_1$ to the edges of $\Gamma$.  Then we extend $\Lab$ to a map from the set of all paths in $\Gamma$ to $S_\circ^*$ in the following natural way.
\begin{itemize}
\item If $e = (x, y)$ is a directed edge labeled $s$, then $\Lab(x, e, y) = s$ and $\Lab(y,e,x) = s^{-1}$.
\item If $\alpha = (x_0, e_1, x_1, \ldots, x_{n-1}, e_n, x_n)$ is a path, then 
$$\Lab(\alpha) = \Lab(x_0, e_1, x_1) \Lab(x_1, e_2, x_2) \cdots \Lab(x_{n-1}, e_n, x_n).$$ 
\end{itemize}

For a path $\alpha$ we define $\ell(\alpha)$, the length of $\alpha$, to be the number of edges traversed by $\alpha$, counting multiplicity.  Equivalently, $\ell(\alpha) = |\Lab(\alpha)|$.

A \emph{plane graph} is a graph which is topologically embedded in $\mathbb R^2$.  A \emph{face} of a plane graph $M$ is the closure of a connected component of $\mathbb R^2 \smallsetminus M$.  Let $F$ be a face of a finite directed plane graph $M$ with edges labeled by elements of $S_1$.  Choosing a base point $x \in \bd F$ and an orientation counterclockwise $(+)$ or clockwise $(-)$, there is a unique circuit which traverses $\bd F$ exactly once, called the \emph{boundary path} and denoted $(\bd F, x, \pm)$.  If all properties of $(\bd F, x, \pm)$ that we care about are preserved after changing its base point and orientation, then we leave these choices out of the notation and write $\bd F$.   The \emph{boundary label} of $F$ is $\Lab(\bd F, x, \pm)$, sometimes denoted by just $\Lab(\bd F)$.  We write $\bd M$ instead of $\bd F$ if $F$ is the unbounded face; from now on, `face' will mean `bounded face' unless otherwise stated.

\begin{definition}
A \emph{van Kampen diagram} over a presentation $\langle S \mid R \rangle$ is a finite, connected, directed plane graph $M$ with edges labeled by elements of $S_1$, such that if $F$ is a face of $M$, then either $\Lab(\bd F) \in R_*$ or $\Lab(\bd F) =_{F(S)} 1$.
\end{definition}

An edge is \emph{essential} if it is labeled by an element of $S$, and \emph{inessential} if it is labeled by 1.  A face $F$ is called \emph{essential} if $\Lab(\bd F) \in R_*$ and \emph{inessential} if $\Lab(\bd F)=_{F(S)} 1$.  If $R$ is cyclically reduced then these cases are mutually exclusive.  A face with boundary label $r \in R$ is called an $r$-face.  We call a van Kampen diagram \emph{bare} if it contains no inessential faces, and \emph{padded} otherwise.

A \emph{subdiagram} of a van Kampen diagram $M$ is a simply connected union of faces of $M$.  If $M$ is a van Kampen diagram and $D$ is a subdiagram of $M$, then we call $D$ \emph{simple} if $\bd D$ is a simple closed curve in the plane.  Likewise, a face $F$ of $M$ is called simple of $\bd F$ is a simple closed curve.  

Let $\alpha$ and $\beta$ be two paths in a van Kampen diagram.  Then we say that $\alpha \cap \beta$ is \emph{trivial} if it contains at most one vertex, and nontrivial otherwise. We say that $\alpha$ and $\beta$ intersect \emph{simply} if $\alpha \cap \beta$ a single subpath of both $\alpha$ and ($\beta$ or the reverse path of $\beta$).  Note that this is \textit{not} the same as saying that $\alpha \cap \beta$ is connected.  We apply this terminology to faces as well.  For example, if we say that $F$ and $\alpha$ intersect simply, it means that there is a choice of base point $x \in \bd F$ such that $(\bd F, x, +)$ and $\alpha$ intersect simply.  If we say that two faces $F$ and $F'$ intersect simply, it means that $(\bd F, x, +)$ and $(\bd F', x, -)$ intersect simply for some $x \in \bd F \cap \bd F'$.

Let $M$ be a van Kampen diagram, and suppose $F$ and $F'$ are distinct faces of $M$.  Then we say that $F$ and $F'$ \emph{cancel} if there exists an edge $e=(x,y)$ in $\bd F \cap \bd F'$ such that $\Lab(\bd F, x, +) = \Lab(\bd F', x, -)$.  A van Kampen diagram is called \emph{reduced} if no two of its faces cancel.  We have the following geometric interpretation of the $C'(\lambda)$ condition, which follows immediately from the definition.

\begin{lemma}\label{SmallCancellationGeometric}
Let $\langle S \mid R \rangle$ be a presentation where $R$ satisfies $C'(\lambda)$, and let $M$ be a van Kampen diagram over $\langle S \mid R \rangle$.  Suppose that $F, F'$ are essential faces of $M$ and $\alpha$ is a common subpath of $\bd F$ and $\bd F'$.  Then either $F$ and $F'$ cancel, or $\ell(\alpha) < \lambda \min (\ell(\bd F), \ell(\bd F'))$.
\end{lemma}

Whenever $G$ is a group generated by $S$, the Cayley graph of $G$ with respect to $S$ is denoted $\Gamma(G,S)$.

\begin{lemma}[van Kampen Lemma]\cite[Chapter V, Section 1]{Lyndon_Schupp}\label{vKL}
Let $G = \langle S \mid R \rangle$ and $w \in S_\circ^*$.  Then $w=_G 1$ if and only if there exists a van Kampen diagram $M$ over $\langle S \mid R \rangle$ and $x \in \bd M$ such that $\Lab(\bd M, x, +)=w$.  Furthermore, given $g \in G$, there exists a combinatorial map $f: M \to \Gamma(G,S)$ preserving labels and orientations of edges, such that $f(x)=g$.  In particular, $f$ does not increase distances, i.e. is $1$-Lipschitz.
\end{lemma}

\subsection{Operations on van Kampen diagrams}\label{vanKampenOps}

Given a van Kampen diagram $M$ over a presentation $\langle S \mid R \rangle$, there are various ways to deform $M$ within the plane to get another van Kampen diagram $M'$.  To check that the resulting graph $M'$ is really a van Kampen diagram, it suffices to show that the operation preserves connectedness and produces a planar embedding of $M'$.  If one also requires that $M'$ is a van Kampen diagram over the \emph{same} presentation, one needs to check that any new faces enclosed by the operation have a boundary label which is either in $R_*$ or equal to the identity in $F(S)$.  In this section we list a few operations on van Kampen diagrams that are needed for the proof of \cref{quasiGeod}.  In our case it will be necessary to keep track of how each operation affects the boundary label $\Lab(\bd M)$, as well as two quantities that we call the signed and unsigned $r$-face counts.

\begin{definition}\label{signedrFaceCount}
Let $M$ be a van Kampen diagram over a presentation $\langle S \mid R \rangle$, where $R$ is cyclically minimal.  Let $r \in R$. Then the (unsigned) \emph{$r$-face count} $\kappa(M,r)$ is the total number of $r$-faces in $M$.  
\end{definition}

\begin{definition}
Let $M$ be a van Kampen diagram over a presentation $\langle S \mid R \rangle$ where $R$ is cyclically minimal, and let  $r \in R$.  Then the \emph{signed $r$-face count} $\sigma(M,r)$ is defined as follows.
\begin{itemize}
\item If $F$ is a face of $M$, then
\[ \sigma(F, r) =
\begin{cases}
1 &\text{ if }\Lab(\bd F, x, +) = r \text{ for some }x \in \bd F \\
-1 &\text{ if } \Lab(\bd F, x, -) = r \text{ for some }x \in \bd F \\
0  &\text{ otherwise.}
\end{cases}
\]
\item $\sigma(M,r) = \sum \{\sigma(F,r) \mid F \text{ is a face of }M\}$.
\end{itemize}
\end{definition}

The assumption that $R$ is cyclically minimal ensures that each face contributes to the signed  or unsigned $r$-face count of at most one $r \in R$.  Note that if $F$ and $F'$ are two faces of $M$ that cancel with each other, then $\sigma(F,r) = -\sigma(F', r)$ for all $r \in R$.

\begin{operation}[Removing an inessential edge]\label{edgeDelete}
Suppose that $e = (x,y)$ is an inessential edge of a van Kampen diagram $M$ over a presentation $\langle S \mid R \rangle$, where $R$ is cyclically reduced and cyclically minimal, and $\Lab(\bd M)$ is cyclically reduced.  Then $e$ is on the boundary of exactly two inessential bounded faces.  There are two possibilities.
\begin{enumerate}[label=\normalfont(\alph*)]
\item If $x \neq y$, contract $e$ to remove it.  This will produce a connected, planar embedding of the new graph.  This changes two inessential faces with labels $1u$ and $1v$ to two inessential faces with labels $u$ and $v$.  Since $R$ is cyclically reduced, this does not affect the $r$-face count for any $r \in R$.
\item If $x=y$,
 delete $e$ to remove it.  Since $e$ is a loop, this will leave the graph connected.  This replaces two inessential faces on either side of $e$ with labels $u1$ and $1v$ with a single inessential face labeled $uv$.  Again since $R$ is cyclically reduced, this operation does not affect $\sigma(M,r)$ for any $r \in R$.
\end{enumerate}
Note that neither (a) nor (b) can introduce new self-intersections in the boundary path of any face of $M$.  Also, since $\Lab(\bd M)$ is cyclically reduced, neither operation affects $\Lab(\bd M)$.
\end{operation}

\begin{operation}[Removing a simple subdiagram with trivial boundary label]\label{diskDelete}
Let $M$ be a van Kampen diagram over $\langle S \mid R \rangle$, where $R$ and $\Lab(\bd M)$ are both cyclically reduced.  Suppose that $M$ contains a simple subdiagram $D$ such that $\bd D$ contains no inessential edges and $\Lab (\bd D)=_{F(S)} 1$.  Then $\bd D = \alpha_+ \alpha_-$, where $\Lab(\alpha_-) = \Lab(\alpha_+)^{-1}$.  We may then remove $D$ by replacing $D$ with a simple inessential face $F$ and deforming $\alpha_+$ onto $\alpha_-$ through the interior of $F$.  This does not affect the boundary label of $M$.
\end{operation}

\begin{figure}[h!]
\centering
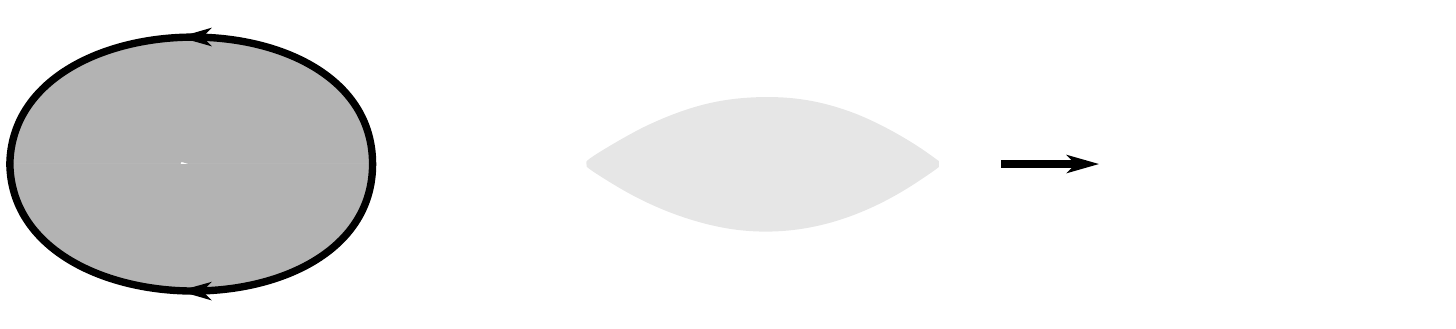
\caption{}
\label{diskDeleteFigure}
\end{figure}

Note that if $F$ and $F'$ are simple faces that intersect simply, and $F$ cancels with $F'$, then $F \cup F'$ is a simple subdiagram of $M$ with trivial boundary label, which may be removed by applying \cref{diskDelete}.  Perhaps surprisingly, \cref{diskDelete} does not always preserve the signed $r$-face count, as the following example shows.

\begin{example}
Figure \ref{nonZeroSumExample} depicts a van Kampen diagram $M$ over the presentation\\ $\langle a, b \mid a^2, aba^{-1}b\rangle$ with boundary label $bb^{-1}$, such that $\sigma(M, aba^{-1}b) = 2$.
\end{example}

\begin{figure}[h!]
\centering
\begingroup%
  \makeatletter%
  \providecommand\color[2][]{%
    \errmessage{(Inkscape) Color is used for the text in Inkscape, but the package 'color.sty' is not loaded}%
    \renewcommand\color[2][]{}%
  }%
  \providecommand\transparent[1]{%
    \errmessage{(Inkscape) Transparency is used (non-zero) for the text in Inkscape, but the package 'transparent.sty' is not loaded}%
    \renewcommand\transparent[1]{}%
  }%
  \providecommand\rotatebox[2]{#2}%
  \newcommand*\fsize{\dimexpr\f@size pt\relax}%
  \newcommand*\lineheight[1]{\fontsize{\fsize}{#1\fsize}\selectfont}%
  \ifx\svgwidth\undefined%
    \setlength{\unitlength}{141.73212126bp}%
    \ifx\svgscale\undefined%
      \relax%
    \else%
      \setlength{\unitlength}{\unitlength * \real{\svgscale}}%
    \fi%
  \else%
    \setlength{\unitlength}{\svgwidth}%
  \fi%
  \global\let\svgwidth\undefined%
  \global\let\svgscale\undefined%
  \makeatother%
  \begin{picture}(1,0.87852991)%
    \lineheight{1}%
    \setlength\tabcolsep{0pt}%
    \put(0,0){\includegraphics[width=\unitlength,page=1]{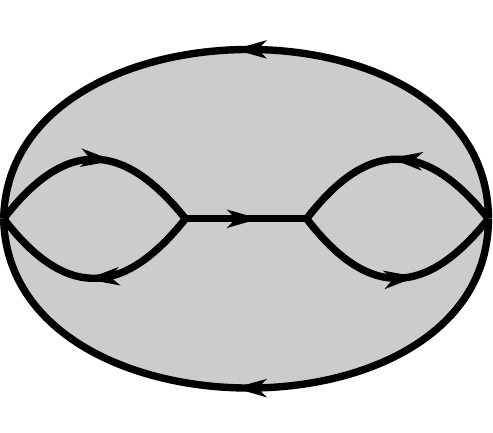}}%
    \put(0.20643582,0.58475347){\color[rgb]{0,0,0}\makebox(0,0)[t]{\lineheight{1.25}\smash{\begin{tabular}[t]{c}$a$\end{tabular}}}}%
    \put(0.20643582,0.23701496){\color[rgb]{0,0,0}\makebox(0,0)[t]{\lineheight{1.25}\smash{\begin{tabular}[t]{c}$a$\end{tabular}}}}%
    \put(0.82253776,0.59609278){\color[rgb]{0,0,0}\makebox(0,0)[t]{\lineheight{1.25}\smash{\begin{tabular}[t]{c}$a$\end{tabular}}}}%
    \put(0.82253776,0.22945544){\color[rgb]{0,0,0}\makebox(0,0)[t]{\lineheight{1.25}\smash{\begin{tabular}[t]{c}$a$\end{tabular}}}}%
    \put(0.50125763,0.81531922){\color[rgb]{0,0,0}\makebox(0,0)[t]{\lineheight{1.25}\smash{\begin{tabular}[t]{c}$b$\end{tabular}}}}%
    \put(0.50125763,0.0064492){\color[rgb]{0,0,0}\makebox(0,0)[t]{\lineheight{1.25}\smash{\begin{tabular}[t]{c}$b$\end{tabular}}}}%
    \put(0.50125763,0.4940391){\color[rgb]{0,0,0}\makebox(0,0)[t]{\lineheight{1.25}\smash{\begin{tabular}[t]{c}$b$\end{tabular}}}}%
    \put(0,0){\includegraphics[width=\unitlength,page=2]{SCAD11.pdf}}%
  \end{picture}%
\endgroup%

\caption{}
\label{nonZeroSumExample}
\end{figure}

However, \cref{diskDelete} does preserve the signed $r$-face count of van Kampen diagrams over $C'(\sfrac{1}{6})$ presentations.  This is because $C'(\sfrac{1}{6})$ presentations are aspherical.  The definition of a spherical van Kampen diagram is the same as that of a van Kampen diagram with $\mathbb R^2$ replaced by $S^2$: in particular, every face is bounded.  A presentation $\langle S \mid R \rangle$ is \emph{aspherical} if  every bare spherical van Kampen diagram over $\langle S \mid R \rangle$ contains a pair of faces that cancel.  The following is a special case of a lemma of Olshanskii.

\begin{lemma}\cite[Lemma 31.1 part 2)]{Olshanskii}\label{ZeroSignedSum}
Let $\langle S \mid R \rangle$ be an aspherical presentation, and suppose that $M$ is a van Kampen diagram over $\langle S \mid R \rangle$ with boundary label $w$, where $w=_{F(S)} 1$.  Then $\sigma(M,r) = 0$ for all $r \in R$.
\end{lemma}

\begin{operation}[Padding a vertex]\label{padVertex}
Suppose that $x$ is a vertex of $M$ which appears twice in the boundary path of some face (bounded or unbounded) of $M$.  Choose $\varepsilon > 0$ small enough so that $B(x, \varepsilon) \subset \mathbb R^2$ contains only the ends of edges incident to $x$.  Now $B(x, \varepsilon) \smallsetminus M$ consists of finitely many connected components: let these be denoted $C_0, C_1,  \ldots, C_k$.  For each $i \in \{0, \ldots, k\}$, insert a clone $x_i$ of $x$ into $C_i$, and connect it to $x$ with an inessential edge.  Then duplicate the edges on either side of $x_i$, attaching the endpoint meant for $x$ to $x_i$ instead: see Figure \ref{padVertexFigure}.  The resulting graph has the same essential faces and boundary path as $M$, and one fewer vertex that is a point of self-intersection of the boundary path of a face.  Each new inessential face has boundary label $1ss^{-1}$ for some $s \in S$.
\end{operation}

\begin{figure}[h!]
\centering
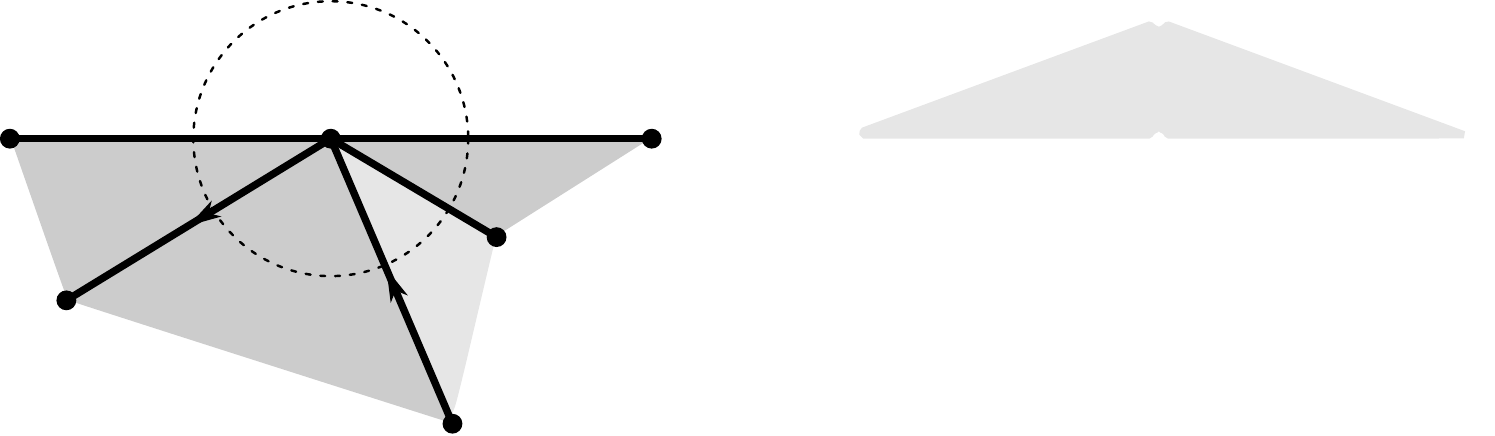
\caption{}
\label{padVertexFigure}
\end{figure}

\begin{operation}[Quotienting simple faces]
Suppose that $G = \langle S \mid R_G \rangle$ and $H = \langle S \mid R_H \rangle$ is a quotient of $G$, so every word in $R_G$ represents the identity element of $H$.  Suppose that $M_G$ is a van Kampen diagram over $\langle S \mid R_G \rangle$.  Let $F$ be a simple face of $M_G$, and let $M_F$ be a chosen van Kampen diagram over $\langle S \mid R_H \rangle$ with boundary label $\Lab(\bd F)$.  Then we may quotient $F$ to a copy of $M_F$ without affecting the boundary label of $M_G$: see Figure \ref{quotientFigure}.  Applying this operation once produces a van Kampen diagram over $\langle S \mid R_G \cup R_H \rangle$.  If $F$ is the last face of $M_G$ with label in $R_G \smallsetminus R_H$, then this results in a van Kampen diagram over $\langle S \mid R_H \rangle$. Thus, if this operation can be applied to every essential face of $M_G$ in sequence, then we obtain a ``quotient van Kampen diagram" $M_H$  over $\langle S \mid R_H \rangle$ with the same boundary label as $M_G$.
\end{operation}

\begin{figure}[h!]
\centering
\begingroup%
  \makeatletter%
  \providecommand\color[2][]{%
    \errmessage{(Inkscape) Color is used for the text in Inkscape, but the package 'color.sty' is not loaded}%
    \renewcommand\color[2][]{}%
  }%
  \providecommand\transparent[1]{%
    \errmessage{(Inkscape) Transparency is used (non-zero) for the text in Inkscape, but the package 'transparent.sty' is not loaded}%
    \renewcommand\transparent[1]{}%
  }%
  \providecommand\rotatebox[2]{#2}%
  \newcommand*\fsize{\dimexpr\f@size pt\relax}%
  \newcommand*\lineheight[1]{\fontsize{\fsize}{#1\fsize}\selectfont}%
  \ifx\svgwidth\undefined%
    \setlength{\unitlength}{458.87210876bp}%
    \ifx\svgscale\undefined%
      \relax%
    \else%
      \setlength{\unitlength}{\unitlength * \real{\svgscale}}%
    \fi%
  \else%
    \setlength{\unitlength}{\svgwidth}%
  \fi%
  \global\let\svgwidth\undefined%
  \global\let\svgscale\undefined%
  \makeatother%
  \begin{picture}(1,0.26083129)%
    \lineheight{1}%
    \setlength\tabcolsep{0pt}%
    \put(0,0){\includegraphics[width=\unitlength,page=1]{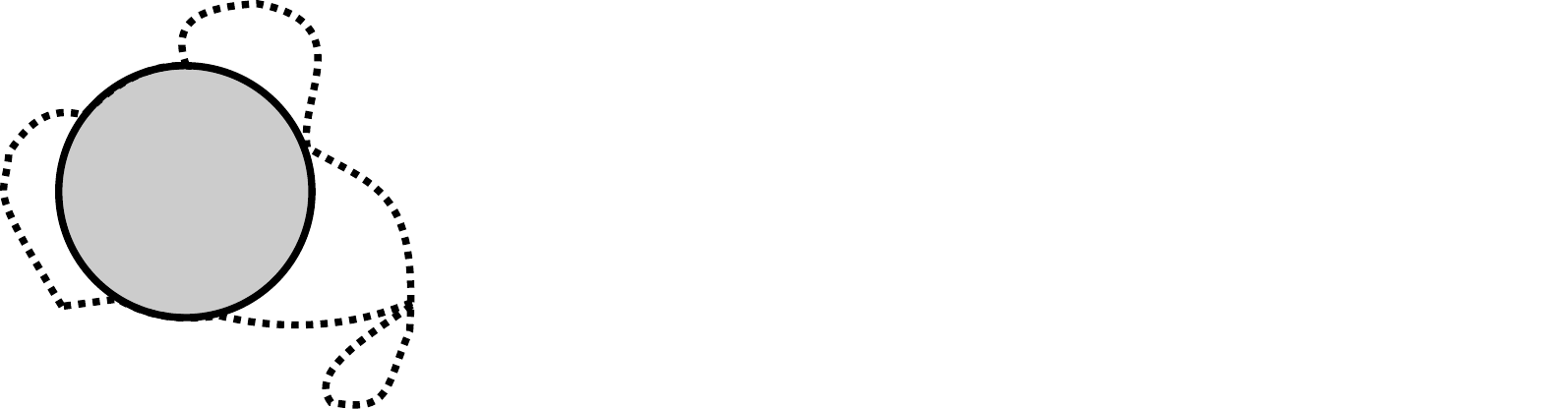}}%
    \put(0.11845292,0.13823232){\color[rgb]{0,0,0}\makebox(0,0)[t]{\lineheight{1.25}\smash{\begin{tabular}[t]{c}$F$\end{tabular}}}}%
    \put(0.0442233,0.23631982){\color[rgb]{0,0,0}\makebox(0,0)[t]{\lineheight{1.25}\smash{\begin{tabular}[t]{c}$M_G$\end{tabular}}}}%
    \put(0,0){\includegraphics[width=\unitlength,page=2]{SCAD17.pdf}}%
    \put(0.95519941,0.09146734){\color[rgb]{0,0,0}\makebox(0,0)[t]{\lineheight{1.25}\smash{\begin{tabular}[t]{c}$M_F$\end{tabular}}}}%
    \put(0,0){\includegraphics[width=\unitlength,page=3]{SCAD17.pdf}}%
  \end{picture}%
\endgroup%

\caption{}
\label{quotientFigure}
\label{M_HFigure}
\end{figure}

\begin{operation}[Excising a subpath of $\bd M$]\label{pathExcise}
Let $M$ be a van Kampen diagram over a presentation $\langle S \mid R \rangle$, where $R$ is cyclically minimal and cyclically reduced.  Let $z \in \bd M$, and suppose we can write $(\bd M, z, +)$ as $\alpha * \beta$, where $\alpha$ and $\beta$ are paths of positive length.  Suppose that $\alpha = \alpha_0 * \rho * \alpha_1$, where $\Lab(\rho)$ is a cyclic shift of $r^{\pm 1}$ for some $r \in R$. Let $x$ be the initial and $y$ the terminal vertex of $\rho$, and suppose $x \neq y$.  Then we may contract $x$ to $y$ through the unbounded face, identifying the two vertices to obtain a new van Kampen diagram $M'$.  Now $M'$ has exactly one new face $F'$, where $(\bd F', x, -) = \rho$, so $M'$ is a van Kampen diagram over  the same presentation $\langle S \mid R \rangle$.  Also, $(\bd M', z, +) = \alpha' * \beta$, where $\alpha' = \alpha_0 * \alpha_1$: see Figure \ref{pathExciseFigure}.  Note that $\rho$ may intersect itself: in that case $\bd F'$ will have self-intersections in $M'$, but this is fine.  The only topological feature of $M$ which is essential to this operation is that $x$ and $y$ are distinct.

\begin{figure}[h!]
\centering
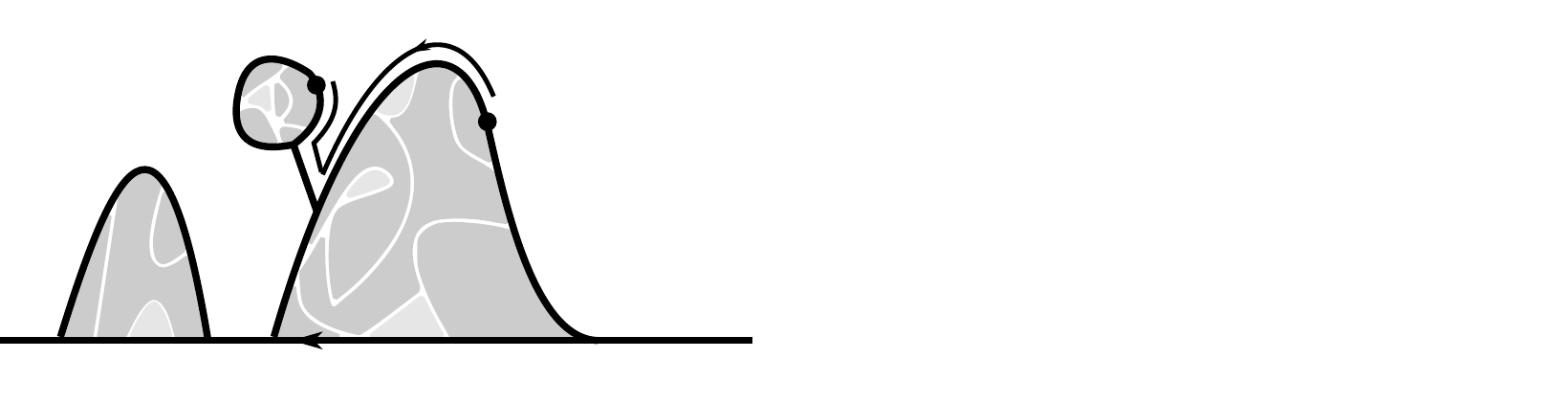
\caption{}
\label{pathExciseFigure}
\end{figure}

Now $\ell(\alpha') = \ell(\alpha)-|r|$, and $\beta$ is unaffected by the operation.  Also, for all $r' \in R$,
\begin{align*}
\kappa(M',r') &=
\begin{cases}
\kappa(M,r')+1 &\text{ if } r'=r\\
\kappa(M,r') &\text{ otherwise.}
\end{cases}\\
\sigma(M',r') &=
\begin{cases}
\sigma(M,r')-1 &\text{ if } r'=r \text{ and }\Lab(\rho) \text{ is a cyclic shift of } r\\
\sigma(M,r')+1 &\text{ if } r'=r \text{ and }\Lab(\rho) \text{ is a cyclic shift of } r^{-1}\\
\sigma(M, r') &\text{ otherwise.}
\end{cases}
\end{align*}
\end{operation}

Note that, since $R$ is cyclically reduced, these last three cases are all distinct.  Indeed, it is an easy exercise to show that if a word $r \in R$ is a cyclic shift of $r^{-1}$, then $r$ is not reduced.

\subsection{Reductions that preserve signed $r$-face counts}\label{vKDsOverC'1-6}

Later we will need to use \cref{perimeterSum}, a  result that applies only to bare, reduced van Kampen diagrams over $C'(\sfrac{1}{6})$ presentations.  At the same time, we would like to apply this result to van Kampen diagrams with signed $r$-face counts that are carefully controlled.  Thus, we need to establish a method of taking a van Kampen diagram over a $C'(\sfrac{1}{6})$ presentation, and making it bare and reduced without affecting the signed $r$-face counts.  In this subsection, we develop such a process, which is encapsulated in \cref{WLOG4}.  We then prove \cref{M_HConstruction}, which allows us to construct certain `quotient' van Kampen diagrams with controlled $r$-face counts. 

\begin{lemma}\label{WLOG1}
Let $M$ be a van Kampen diagram over $\langle S \mid R \rangle$ such that $R$ is cyclically minimal and cyclically reduced, and $\Lab(\bd M)$ is cyclically reduced.  Then there exists a van Kampen diagram $M'$ such that all of the following conditions hold.
\begin{enumerate}[label=\normalfont(\alph*)]
\item $\Lab(\bd M') = \Lab(\bd M)$.
\item $\sigma(M',r) = \sigma(M,r)$ for all $r \in R$.
\item Every inessential face of $M'$ has boundary label $ss^{-1}$ or $1ss^{-1}$ for some $s \in S$.
\item All inessential edges of $M'$ are loops.
\end{enumerate}
\end{lemma}
\begin{proof}
Let $I$ be the set of all inessential faces in $M$ whose boundary labels are not equal to $1ss^{-1}$ or $ss^{-1}$ for some $s \in S$.  Let $F \in I$.  If $\bd F$ consists of a single inessential edge loop, we simply contract this loop to remove $F$; it is easy to see that this preserves the boundary label as well as the signed and unsigned face counts of $M$, and whether $M$ satisfies (c) or (d).  Therefore we assume that $\bd F$ is at least two edges long.  Now we repeatedly pad vertices of $\bd F$ (\cref{padVertex}) until $F$ is simple. Since each inessential face added in the process has boundary label $1ss^{-1}$ for some $s \in S$, this does not increase $|I|$.  

We claim that, without loss of generality, we may assume that $\bd F$ contains no inessential edges.  Suppose that $\bd F$ contains an inessential edge $e$.  Since $\bd M$ and $R$ are both cyclically reduced, $e$ lies on the boundary of exactly two inessential, bounded faces, one of which is $F$: call the other one $F'$.  Then for some $u, u' \in S_\circ^*$ we have $\Lab(F) = 1u$ and $\Lab(F') = 1u'$.  We know that $e$ must have distinct endpoints since $\bd F$ is a simple closed curve and $\ell(\bd F) \geq 2$.  Therefore we may remove $e$ using \cref{edgeDelete} (a).  This changes the boundary label of $F$ from $1u$ to $u$, and the boundary label of $F'$ from $1u'$ to $u'$.  Thus  it does not change whether or not $F$ or $F'$ is a member of $I$.  Therefore removing $e$ does not change $|I|$, and without loss of generality we may assume that $\bd F$ contains no inessential edges.  

Since $F$ is simple, $\Lab(\bd F) =_{F(S)} 1$, and $\bd F$ contains no inessential edges, we may remove $F$ using \cref{diskDelete}.  This reduces $|I|$ by 1.  Since $\Lab(\bd M)$ is cyclically reduced, none of the previous operations affect $\Lab(\bd M)$.  Since only inessential faces were removed, and $R$ is cyclically reduced, $\sigma(M,r)$ is also preserved for all $r \in R$.  Repeating this process, we obtain a diagram $M'$ for which (a) and (b) hold and $|I| = 0$, i.e. such that (a)-(c) hold.  At this point we may repeatedly apply \cref{edgeDelete} (a) to remove all inessential edges of $M'$ with distinct endpoints, so that (d) holds in $M'$.  Reasoning as in the previous paragraph, one can see that this does not interfere with conditions (a)-(c).  Thus (a)-(d) hold in $M'$, finishing the construction. 
\end{proof}

\begin{lemma}\label{WLOG2}
Let $G$ be a group given by presentation $\langle S \mid R \rangle$, where $R$ is cyclically minimal and cyclically reduced, and $s \neq_G 1$ for any $s \in S$.  Let $M$ be a van Kampen diagram over $\langle S \mid R \rangle$ such that $\Lab(\bd M)$ is cyclically reduced.  Then there exists a van Kampen diagram $M'$ over $\langle S \mid R \rangle$ such that all of the following conditions hold.
\begin{enumerate}[label=\normalfont(\alph*)]
\item $\Lab(\bd M') = \Lab(\bd M)$.
\item $\sigma(M',r) = \sigma(M, r)$ for all $r \in R$.
\item Every inessential face of $M'$ is contained in a simple subdiagram whose boundary label is equal to $ss^{-1}$ for some $s \in S$.
\end{enumerate}
\end{lemma}

\begin{proof}
We may assume that we have a van Kampen diagram $M'$ that satisfies (a)-(d) of \cref{WLOG1}.  We prove here that, in the presence of the assumption that $s \neq_G 1$ for all $s \in S$, it follows that $M'$ also satisfies conclusion (c) of the current lemma.

Let $F$ be an inessential face of $M'$.  There are two cases: either $\Lab(\bd F) = ss^{-1}$ or $\Lab(\bd F) = 1ss^{-1}$ for some $s \in S$.  

If $\Lab(\bd F) = ss^{-1}$, then since $s \neq_G 1$, we have that $F$ is simple.  Thus $F$ itself is a simple subdiagram of $M'$ which contains $F$ and has boundary label $ss^{-1}$.

Suppose on the other hand that $\Lab(\bd F) = 1aa^{-1}$, where $a \in S$.  Let $e$ be the inessential edge of $\bd F$.  Since $\bd M'$ and $R$ are both cyclically reduced, $e$ lies on the boundary paths of exactly two inessential, bounded faces, one of which is $F$: call the other one $F'$. Since $F'$ is an inessential face of $M'$ with an inessential edge on its boundary path, we have that $\Lab(F') = 1bb^{-1}$ for some $b \in S$.  Now $e$ is a loop since $M'$ satisfies (d) of \cref{WLOG1}.  Since $s \neq_G 1$, the endpoints of each of the $a$-labeled edges of $\bd F$ are distinct: similarly for the $b$-labeled edges of $\bd F'$.  Therefore $F \cup F'$ takes the form depicted in Figure \ref{Bigon}, allowing that the roles of $F$ and $F'$ may be switched.

\begin{figure}[h!]
\centering
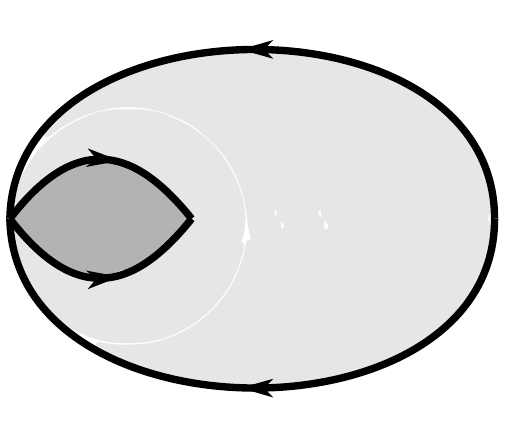
\caption{}
\label{Bigon}
\end{figure}

Notice that in Figure \ref{Bigon}, $F \cup F'$ is enclosed in a simple subdiagram with boundary label $aa^{-1}$ (or $bb^{-1}$, if the roles of $F$ and $F'$ are switched).  This finishes the second case, thus (c) holds for $M'$ and we are done.
\end{proof}

\begin{corollary}\label{WLOG3}
Let $G$ be a group given by an aspherical presentation $\langle S \mid R \rangle$, where $R$ is cyclically minimal and cyclically reduced, and $s \neq_G 1$ for all $s \in S$.  Let $M$ be a van Kampen diagram over $\langle S \mid R \rangle$ such that $\Lab(\bd M)$ is cyclically reduced.  Then there exists a van Kampen diagram $M'$ such that all of the following conditions hold. 
\end{corollary}
\begin{enumerate}[label=\normalfont(\alph*)]
\item $\Lab(\bd M') = \Lab(\bd M)$.
\item $\sigma(M', r) = \sigma(M, r)$ for all $r \in R$.
\item $M'$ is bare.
\end{enumerate}
\begin{proof}
We may assume that $M'$ satisfies (a)-(c) of \cref{WLOG2}.  Now all inessential faces of $M'$ are contained in simple subdiagrams of $M'$ with boundary label $ss^{-1}$ for some $s \in S$.  Thus we may make $M'$ bare by repeatedly applying \cref{diskDelete}.  \cref{diskDelete} always preserves the boundary label of a van Kampen diagram, so (a) holds.  Since $\langle S \mid R \rangle$ is aspherical, it follows from \cref{ZeroSignedSum} that each application of \cref{diskDelete} preserves $\sigma(M', r)$ for all $r \in R$.  Thus (a)-(c) hold for $M'$, and we are done.
\end{proof}

Often one would like to take a van Kampen diagram $M$ which is not reduced, and reduce it using \cref{diskDelete}.  However, the canceling faces may not be simple, or may not intersect each other simply.  A common solution is to pad the van Kampen diagram with inessential faces.  However, if $M$ is a van Kampen diagram over a $C'(\sfrac{1}{6})$ presentation, then $M$ is topologically well-behaved enough to perform this operation without the use of inessential faces.  We make this statement precise in the following two lemmas.  The second is a consequence of the first, which is the famous Greendlinger Lemma. 

\begin{lemma}[Greendlinger Lemma] \cite[Chapter V, Theorem 4.5]{Lyndon_Schupp}
Let $M$ be a bare and reduced van Kampen diagram over a cyclically reduced $C'(\lambda)$ presentation, where $\lambda \leq \sfrac{1}{6}$, such that $M$ has at least one bounded face and $\Lab(\bd M)$ is cyclically reduced.  Then there exists a face $F$ of $M$ such that $\bd F$ and $\bd M$ share a common subpath of length more than $\frac{1}{2} \ell(\bd F)$.
\end{lemma}

\begin{lemma}
\label{facesAreDisks}
Let $M$ be a bare van Kampen diagram over a cyclically reduced $C'(\sfrac{1}{6})$ presentation.  Then
\begin{enumerate}[label=\normalfont(\alph*)]
\item If $M$ is reduced, then every face of $M$ is simple.
\item If $M$ is reduced, then every two faces of $M$ that intersect nontrivially also intersect simply.
\item If $M$ is not reduced, then there exists a pair of faces that cancel and intersect simply.
\end{enumerate}
\end{lemma}

\begin{proof}
We refer the reader to \cite[Chapter V, Lemma 4.1]{Lyndon_Schupp} for the proof of part (a).  For part (b), suppose that $M$ is a counterexample with the minimum number of faces, and that $F$ and $F'$ are two faces of $M$ that do not intersect simply, i.e. such that $\bd F$ intersects $\bd F'$ in more than one maximal common subpath. Then $\bd F$ and $\bd F'$ together enclose at least one simple subdiagram of $M$, call it $D$.  Since $M$ is reduced, so is $D$.  By the Greendlinger Lemma, there exists a face $E$ of $D$ such that $\bd E$ intersects $\bd D$ in a subpath of length at least $\frac{1}{2}\ell(\bd E)$.  Therefore $\bd E$ intersects one of $F$ or $F'$, say $F$, in a common subpath of length at least $\frac{1}{4}\ell(\bd E)$.  But then $E$ and $F$ cancel, contradicting the assumption that $M$ is reduced.

For part (c), suppose that $M$ is a counterexample with the minimum number of faces.  Then $M$ is not reduced, and there are two faces $F$ and $F'$ that cancel but do not intersect simply.  Therefore $\bd F$ and $\bd F'$ together enclose a simple subdiagram $D$.  Again $D$ must be reduced, this time by minimality of $M$.  By an argument similar to the one in the preceding paragraph, there is a face $E$ of $D$ that cancels with $F$.  By assumption, $E$ and $F$ cannot intersect simply.  But then $D \cup F$ is a subdiagram of $M$ that is a counterexample with strictly fewer faces than $M$, since it does not include $F'$.  This contradicts minimality of $M$, finishing the proof.
\end{proof}

One can then use \cref{facesAreDisks} to prove the following corollary.

\begin{corollary}\label{corFacesAreDisks}
Let $G$ be a group given by presentation $\langle S \mid R \rangle$, where $R$ is cyclically reduced and satisfies $C'(\sfrac{1}{6})$.  Then for every $r \in R$, if $u$ is a subword of an element of $\{r\}_*$, then $u \neq_G 1$.  In particular,
\begin{enumerate}[label=\normalfont(\alph*)]
\item For all generators $s \in S$, if $s \not \in R$ then $s \neq_G 1$.
\item If $M$ is a bare van Kampen diagram over $\langle S \mid R \rangle$, then every face of $M$ is simple.
\item If $M$ is a bare van Kampen diagram over $\langle S \mid R \rangle$ which is not reduced, then there exists a pair of cancelling faces that are simple and intersect simply.
\end{enumerate} 
\end{corollary}

\begin{proof}
Suppose otherwise, and choose $u \in S_\circ^*$ to be a word of minimum length which is a proper subword of $\{r\}_*$ for some $r \in R$.  Without loss of generality, suppose that $u$ is a prefix of $r$, so that $r = uv$ for some $v \in S_\circ^*$.  Clearly $v =_G 1$, so by minimality of $u$ we have that $|v| \geq |u|$, hence $|u| \leq \frac{1}{2}|r|$.

Let $M$ be a reduced van Kampen diagram over $\langle S \mid R \rangle$ such that $\Lab(\bd M) = u$.  Since $R$ is cyclically reduced, so is $u$: in particular, $M$ has at least one bounded face.  By the Greendlinger Lemma, there exists a face $F$ of $M$ such that $F$ shares a common subpath of length more than $\frac{1}{2}\ell(\bd F)$ with $\bd M$.  Let the label of this common subpath be $w$.  Then $w$ is a piece of $r$ and $\Lab(\bd F)$, of length more than $\frac{1}{2}\ell(\bd F)$.  By the $C'(\sfrac{1}{6})$ condition, $\Lab(\bd F) = r$.  But then $|w| > \frac{1}{2}|r| \geq |u|$.  Since $w$ is the label of a subpath of $\bd M$, this is a contradiction.

Conclusions (a) and (b) follow directly.  Part (c) follows from part (b) of the current lemma and \cref{facesAreDisks} (c). 
\end{proof}

\begin{lemma}\label{WLOG4}
Let $M$ be a van Kampen diagram over a $C'(\sfrac{1}{6})$ presentation $\langle S \mid R \rangle$, where $R$ is cyclically minimal and cyclically reduced, and $|r| \geq 2$ for all $r \in R$.  Then there exists a van Kampen diagram $M'$ over $\langle S \mid R \rangle$ such that
\begin{enumerate}[label=\normalfont(\alph*)]
\item $\Lab(\bd M') = \Lab(\bd M)$.
\item $\sigma(M', r) = \sigma(M, r)$ for all $r \in R$.
\item $M'$ is bare and reduced.
\end{enumerate}
\end{lemma}
\begin{proof}
We may assume that $M'$ satisfies (a)-(c) of \cref{WLOG3}.  Thus we only have to show that it is possible to transform $M'$ so that it is reduced, while preserving the boundary label and signed $r$-face count for each $r \in R$, and without adding any inessential faces.

Suppose that $M'$ is not reduced.  Since $M'$ is bare, by \cref{corFacesAreDisks} there exist two simple faces $F$ and $F'$ that cancel and intersect simply.  Thus $F \cup F'$ is a simple subdiagram of $M$ with trivial boundary label. Now remove $F \cup F'$ with \cref{diskDelete}.  Since $\langle S \mid R \rangle$ is $C'(\sfrac{1}{6})$, and therefore aspherical, this operation preserves $\sigma(M', r)$ for all $r \in R$.  Repeating, we end up with a reduced van Kampen diagram.
\end{proof}

\begin{lemma}\label{M_HConstruction}
Let $G,H$ be groups given by presentations
\begin{align*}
G &= \langle S \mid R_G \rangle \\
H &= \langle S \mid R_H \rangle
\end{align*}
where $\langle S \mid R_H \rangle$ is a cyclically reduced $C'(\sfrac{1}{6})$ presentation, and $|r_H| \geq 2$ for all $r_H \in R_H$.  Suppose that $r_G =_H 1$ for all $r_G \in R_G$, so $H$ is a quotient of $G$.  Let $M_G$ be a van Kampen diagram over $\langle S \mid R_G \rangle$, and for each face $F$ of $M_G$, let $M_F$ be a van Kampen diagram over $\langle S \mid R \rangle$ with boundary label $\Lab(\bd F)$.  Then there exists a ``quotient van Kampen diagram" $M_H$ over $\langle S \mid R_H \rangle$ such that
\begin{enumerate}[label=\normalfont(\alph*)]
\item $\Lab(\bd M_G) = \Lab(\bd M_H)$.
\item $\sigma(M_H, r) = \sum\{\sigma(M_F, r) \mid F \text{ is an essential face of } M_G \}$.
\item $M_H$ is bare and reduced.
\end{enumerate}
\end{lemma}
\begin{proof}
Start with $M_G$. By repeatedly padding vertices, we may assume that all essential faces of $M_G$ are simple. 

Now take an essential disk face $F$ of $M_G$, and quotient it to $M_F$.  This may introduce self-intersections among essential faces in $M_G$.  Pad vertices again until all essential faces of $M_G$ are simple, and repeat as many times as necessary to quotient all essential faces that were originally in $M_G$. Since padding vertices and quotienting simple faces preserve the boundary label, we obtain a van Kampen diagram $M_H$ over $\langle S \mid R_H \rangle$, possibly with many inessential faces, such that $\Lab(\bd M_H) = \Lab(\bd M_G)$.  Thus (a) holds.  In addition, for all $r \in R_H$,
$$\sigma(M_H,r) = \sum\{\sigma(M_F, r) \mid F \text{ is a face of } M_G\} \, ,$$ 
so (b) holds as well.  

Note that we do \emph{not} require $R_G$ to by cyclically reduced for any of the previous steps to work.  However, $R_H$ \textit{is} cyclically reduced, thus by \cref{WLOG4} we may ensure that (c) holds, without interfering with conditions (a) or (b).
\end{proof}

\subsection{A technical lemma}\label{technicalLemma}

This section is devoted to proving the following lemma.  In essence it is similar to \cite[Lemma 5.10]{Olshanskii_Osin_Sapir}, but for our purposes we need the more general version stated here.  In order to avoid constantly reiterating the assumptions, the notation used in this lemma will be `globally fixed' for this section.  Thus until the next section, $G$ will always refer to the group with presentation given in \cref{quasiGeod}, etc.  Any new notation introduced in the body of this section will also remain fixed until the beginning of the next section. 

\begin{lemma}\label{quasiGeod}
Let $\lambda$ be a real number, where $0 < \lambda < \sfrac{1}{12}$.  Let $\{\ell_i \mid i \in \mathbb N\}$ be a set of positive integers, where each $\ell_i \geq 2$.  Let $S$ be a finite set.  Let
\begin{align*}
U = \{u_i \mid i \in \mathbb N\} \subset S_\circ^*&&
V = \{v_i \mid i \in \mathbb N\} \subset S_\circ^*
\end{align*}
be languages, and let $\tilde u \in S_\circ^*$ be a word, such that the following conditions are satisfied for all $i,i' \in \mathbb N$.
\begin{enumerate}[label=\normalfont(\alph*)]
\item $U \cup V$ is cyclically minimal and cyclically reduced, and satisfies $C'(\lambda)$.
\item $2 \leq |u_i| \leq |v_i|$.
\item If $p$ is a piece of $\tilde u$ and $u_i$, then $|p| < \lambda|u_i|$, and the same statement holds if $u_i$ is replaced with $v_i$.
\item If $u_i = u_{i'}$, $v_i= v_{i'}$, or $u_i = v_{i'}$, then $i=i'$.
\end{enumerate}
Now let
\begin{align*}
G &= \langle S \mid R_G \rangle := \langle S \mid [s, u_i], u_i^{\ell_i}, u_iv_i^{-1}: s \in S, i \in \mathbb N \rangle\\
H &= \langle S \mid R_H \rangle := \langle S \mid U \cup V \rangle = \langle S \mid u_i, v_i : i \in \mathbb N \rangle \, .
\end{align*}
Let $(k_i)$ be a sequence of integers where $|k_i| \leq \ell_i/2$ for all $i \in \mathbb N$,  and $k_i = 0$ for all but finitely many $i \in \mathbb N$.  Let $u \in S_\circ^*$ be a word of the form
$$u = \tilde u \prod_{i=0}^\infty {u_i^{k_i}}_{\, .}$$

Then $u$ is $\left({\frac{3}{1-12\lambda}}_{\, ,} 0 \right)$-quasigeodesic in $G$.
\end{lemma}

Note that if $U \cup V$ is $C'(\lambda)$, then so is $(U \cup V \cup \{u_i^{-1}\}) \smallsetminus \{u_i\}$.  Therefore assume without loss of generality that all $k_i$ are nonnegative.

Let $w$ be a geodesic representative of $u$ in $G$.  Then $uw^{-1} =_G 1$, so by the van Kampen Lemma, there exists a van Kampen diagram $M_G$ with $\Lab(\bd M_G) = uw^{-1}$.  

\begin{lemma}\label{specificM_HConstruction}
There exists a van Kampen diagram $M_H$ over $\langle S \mid U \cup V \rangle$ such that
\begin{enumerate}[label=\normalfont(\alph*)]
\item $M_H$ is bare and reduced.
\item $\bd M_H = \alpha * \beta$, where $\Lab(\alpha) = u$ and $\Lab(\beta) = w^{-1}$.
\item $\sigma(M_H, u_i)+\sigma(M_H, v_i) \equiv 0 \mod \ell_i$ for all $i \in \mathbb N$.
\item $\sigma(M_H, u_i) \equiv 0 \mod \ell_i$ for all $i \in \mathbb N$ such that $u_i = v_i$.
\end{enumerate}
\end{lemma}
\begin{proof}
Each face $F$ of $M_G$ has boundary label equal to either $[s,u_i]$, $u_i^{\ell_i}$, or $u_iv_i^{-1}$.  Each of these words represents the trivial elment of $H$.  For each face $F$ of $M_G$, choose a van Kampen diagram $M_F$ over $\langle S \mid R_H \rangle$, of one of forms depicted in Figure \ref{M_FFigure}.

\begin{figure}[h!]
\centering
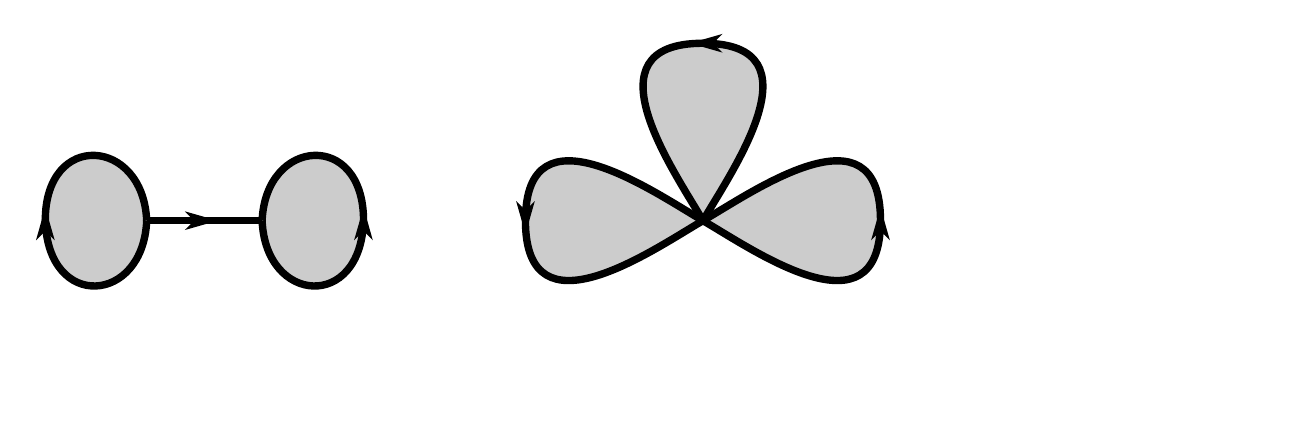
\caption{}
\label{M_FFigure}
\end{figure}

Applying \cref{M_HConstruction}, there exists a bare, reduced van Kampen diagram $M_H$ with $\Lab(M_H) = \Lab(M_G) = uw^{-1}$, such that for all $i \in \mathbb N$,
$$\sigma(M,u_i)+\sigma(M,v_i) = \sum\{\sigma(M_F, u_i)+\sigma(M_F,v_i) \mid F \text{ is a face of } M_G\} \, .$$
Thus (a) and (b) hold.  Now notice that for all $M_F$ depicted in Figure \ref{M_FFigure}, $\sigma(M_F, u_i)+\sigma(M_F, v_i) \equiv 0 \mod \ell_i$.  Also, if $u_i = v_i$, then $\sigma(M_F, u_i) \equiv 0 \mod \ell_i$.  Thus (c) and (d) hold as well.
\end{proof}

Let $M_H$ be the van Kampen diagram from \cref{specificM_HConstruction}.  Let $k = \sum_{i \in \mathbb N} k_i$.  Then we may write
$$\alpha = \tilde \alpha * \alpha_0 * \dots * \alpha_k * \beta$$
where $\Lab(\tilde \alpha) = \tilde u$, and for all $j \in \{0, \ldots, k\}$ we have $\Lab(\alpha_j) = u_i$ for some $i \in \mathbb N$.
\begin{lemma}\label{M_H'Construction}
There exists a van Kampen diagram $M_H'$ over $\langle S \mid U \cup V \rangle$ and natural numbers $\{h_i \mid i \in \mathbb N\}$ satisfying all of the following conditions for all $i \in \mathbb N$.
\begin{enumerate}[label=\normalfont(\alph*)]
\item $M_H'$ is bare and reduced.
\item $\kappa(M_H', u_i) = \kappa(M_H, u_i)-h_i$.
\item $\bd M_H' = \alpha' * \beta'$, where $\Lab(\alpha') = \tilde u \prod_{i=0}^\infty u_i^{k_i-h_i}$ and $\Lab(\beta') = w$.
\item No face $F$ of $M_H'$ intersects $\alpha'$ in a common subpath of length at least $2\lambda \ell(\bd F)$.
\item $0 \leq h_i \leq k_i \leq \ell_i/2$.
\end{enumerate}
\end{lemma}

\begin{proof}
If $M_H$ already satisfies (d), then all conditions are satisfied by setting $M_H' = M_H$ and $h_i = 0$ for all $i \in \mathbb N$.  Therefore suppose that $M_H$ does not satisfy (d), i.e. there exists a face $F$ of $M_H$ such that $\bd F$ intersects $\alpha$ in a common subpath of length at least $2\lambda \ell(\bd F)$.  Then there must be a common subpath of $\bd F$ and $\tilde \alpha$ or $\alpha_j$ for some $j \in \{0, \ldots, k\}$, of length at least $\lambda \ell(\bd F)$.  The former possibility is excluded by condition (c) of \cref{quasiGeod}.  Thus $\bd F$ intersects $\alpha_j$ in a common subpath of length at least $\lambda \ell(\bd F)$ for some $j \in \{0, \ldots, k\}$.  Call this common subpath $\gamma$.  Since $\Lab(\alpha_j) = u_i$ for some $i \in \mathbb N$, by the $C'(\lambda)$ condition we have that $\Lab(\bd F) = u_i$ as well.

Now apply \cref{pathExcise} to excise $\alpha_j$ from $\alpha$.  Let $F'$ be the new $u_i$-face created by this operation.  Then $\gamma$ is a common subpath of $F$ and $F'$ of length at least $\lambda \ell(\bd F) = \lambda \ell(\bd F')$, so $F$ and $F'$ cancel.  Since $M_H$ was reduced, $F$ and $F'$ are the only pair of faces that cancel at this stage.  Therefore by \cref{corFacesAreDisks} (c), $F$ and $F'$ are simple and intersect simply.  Thus $F \cup F'$ is a simple subdiagram of $M$ with trivial boundary label, which we may remove with \cref{diskDelete}.  

Let $\hat M_H$ be the van Kampen diagram obtained in this way.  Then clearly $\hat M_H$ satisfies (a).  We added one $u_i$-face and removed two, so $\kappa(\hat M_H, u_i) = \kappa(M_H, u_i)-1$.  Since $u_i \neq u_j$ whenever $i \neq j$, no $u_j$-face counts were affected for any $j \neq i$.  Therefore (b) is satisfied with $h_i=1$.  Now after excising $\alpha_j$ the boundary path becomes $\hat \alpha * \beta := \tilde \alpha * \alpha_0 * \dots * \alpha_{j-1} * \alpha_{j+1} * \dots * \alpha_k * \beta$.  Removing $F \cup F'$ does not change the boundary label, so (c) is satisfied with $h_i=1$.  Because of this, we may iterate the process.  By construction, $\hat M_H$ has one fewer face than $M_H$ which fails to satisfy (d).  Therefore repeat as many times as there are faces in $M_H$ failing to satisfy (d) to get $M_H'$.  Each such face must be a $u_i$-face for some $i \in \mathbb N$, so for each $i \in \mathbb N$, let $h_i$ be the number of $u_i$-faces in $M_H$ failing to satisfy (d).  Since the boundary label becomes shorter at each step, by (c) it follows that $h_i \leq k_i$ for all $i \in \mathbb N$.  Therefore $M_H'$ satisfies (e), and we are done.
\end{proof}

For the next step in the proof, the following ad hoc lemma is useful.

\begin{lemma}\label{alphaFSingleSubpath}
Let $M$ be a bare, reduced van Kampen diagram over a cyclically reduced $C'(\sfrac{1}{6})$ presentation.  Let $\alpha$ be a subpath of $\bd M$ such that no face $F$ of $M$ intersects $\alpha$ in a common subpath of length at least $\frac{1}{4}\ell(\bd F)$.  Then every face of $M$ that intersects $\alpha$ nontrivially, intersects $\alpha$ simply.
\end{lemma}

\begin{proof}
Suppose that $F$ is a face of $M$ such that $\bd F$ shares more than one vertex with $\alpha$, but $\bd F$ does not intersect $\alpha$ simply.  Then there exist subpaths of $\alpha$ and $\bd F$ that together enclose a simple subdiagram $D$ of $M$.  Since $M$, and therefore $D$, is reduced, by the Greendlinger Lemma there exists a face $F'$ of $D$ such that $\bd F'$ shares a common subpath of length at least $\frac{1}{2}\ell(\bd F')$ with $\bd D$.  Thus $\bd F'$ intersects either $\bd F$ or $\alpha$ in a common subpath of length at least $\frac{1}{4}\ell(\bd F')$.  The latter possibility is ruled out by assumption, so $F$ cancels with $F'$ by the $C'(\sfrac{1}{6})$ condition.  But this contradicts our assumption that $M$ is reduced.
\end{proof}

\begin{definition}
Let $M$ be a van Kampen diagram over a presentation $\langle S \mid R \rangle$.  Then the \emph{perimeter sum} of $M$, denoted $\PS(M)$, is defined by
$$\PS(M) = \sum\{\ell(\bd F) \mid F \text{ is a face of } M\} \, .$$
Note that if $M$ is bare, then 
$$\PS(M) = \sum_{r \in R} |r|\kappa(M,r) \, .$$
\end{definition}

To obtain bounds on $\ell(\alpha')$ in terms of $\ell(\beta)$, and on $\ell(\alpha)$ in terms of $\ell(\alpha')$, we use the following fact about van Kampen diagrams over $C'(\sfrac{1}{6})$ presentations.  It is the final puzzle piece in the proof.

\begin{lemma}\cite[Lemma 3.8]{Olshanskii_Osin_Sapir}\label{perimeterSum}
Let $M$ be a bare and reduced van Kampen diagram over a cyclically reduced $C'(\lambda)$ presentation, where $\lambda \leq \sfrac{1}{6}$.  Then $(1-6\lambda)\PS(M) \leq \ell(\bd M)$.
\end{lemma}

With this in mind, we resume our proof.

\begin{lemma}\label{ineq1}
$\ell(\alpha') < 2\ell(\beta)$.
\end{lemma}
\begin{proof}
Note that an edge of $\alpha'$ is shared by the boundary path of some face of $M_H'$ if and only if it is not also an edge of $\beta'$.  We have by \cref{M_H'Construction} (d) that no face $F$ intersects $\alpha'$ in a common subpath of length at least $2\lambda \ell(\bd F) < \frac{1}{6} \ell(\bd F) < \frac{1}{4}\ell(\bd F)$.  Therefore by \cref{alphaFSingleSubpath}, every face whose boundary path shares an edge with $\alpha'$ intersects $\alpha'$ in a single common subpath.  Thus
$$\PS(M_H) > \tfrac{1}{2\lambda}\ell(\alpha' \smallsetminus \beta') \geq  \tfrac{1}{2\lambda}\left(\ell(\alpha') - \ell(\beta')\right) = \tfrac{1}{2\lambda} \left(\ell(\alpha')-\ell(\beta')\right)$$
On the other hand, $\ell(\bd M_H') = \ell(\alpha')+\ell(\beta')$.  Thus by \cref{perimeterSum},
\begin{align*}
\tfrac{1}{2\lambda}\left(\ell(\alpha') - \ell(\beta')\right) < \PS(M_H') &\leq \tfrac{1}{1-6\lambda}\ell(\bd M_H') = \tfrac{1}{1-6\lambda}\left(\ell(\alpha')+\ell(\beta')\right)\\
(1-6\lambda)(\ell(\alpha')-\ell(\beta')) &< 2\lambda(\ell(\alpha')+\ell(\beta'))\\
(1-8\lambda)\ell(\alpha') &< (1-4\lambda)\ell(\beta')\\
\ell(\alpha') &< \tfrac{1-4\lambda}{1-8\lambda}\ell(\beta') < 2\ell(\beta') = 2\ell(\beta) \, ,
\end{align*}
since $0 < \lambda < \sfrac{1}{12}$.
\end{proof}

\begin{lemma}\label{ineq2}
$\PS(M_H) \geq 2(\ell(\alpha)-\ell(\alpha'))$.
\end{lemma}
\begin{proof}
Let $I = \{i \in \mathbb N \mid u_i = v_i\}$.  By \cref{specificM_HConstruction}, if $i \in I$ then $\sigma(M_H, u_i) \equiv 0 \mod \ell_i$.  If $i \not \in I$, then $\sigma(M_H, u_i)+\sigma(M_H, v_i) \equiv 0 \mod \ell_i$. Note that $\kappa(M_H,u_i)+\kappa(M_H, v_i) \geq \kappa(M_H, u_i) \geq h_i$ by \cref{M_H'Construction} (b).  Since $h_i \leq \ell_i/2$, it follows that there are at least $2h_i$ faces in $M_H$ with boundary label either $u_i^{\pm 1}$ or $v_i^{\pm 1}$.  If $u_i = v_i$, this says that $\kappa(M_H, u_i) \geq 2h_i$.  If $u_i \neq v_i$, this means $\kappa(M_H, u_i)+\kappa(M_H, v_i) \geq 2h_i$.  Therefore
\begin{align*}
\PS(M_H) &= \sum_{r \in R_H}|r|\kappa(M_H, r)\\
&=\sum_{i \in I}|u_i|\kappa(M_H, u_i) + \sum_{i \not \in I}(|u_i|\kappa(M_H, u_i)+|v_i|\kappa(M_H,v_i))\\
&\geq \sum_{i \in I}|u_i|\kappa(M_H, u_i) +\sum_{i \not \in I} |u_i|(\kappa(M_H, u_i)+\kappa(M_H,v_i))\\
&\geq \sum_{i \in I} 2h_i|u_i| + \sum_{i \not \in I} 2h_i|u_i| = \sum_{i \in \mathbb N} 2h_i|u_i| = 2(\ell(\alpha)-\ell(\alpha')) \, ,
\end{align*}
where the last equality follows from \cref{M_H'Construction} (c).
\end{proof}

Now we are ready to prove \cref{quasiGeod}.
\begin{proof}[Proof of \cref{quasiGeod}]
Continuing to use the terminology and notation built up in this section, since $w$ is a geodesic representative of $u$, $|u| = \ell(\alpha)$, and $|w| = \ell(\beta)$, it suffices to prove that $\ell(\alpha) < \frac{3}{1-12\lambda}\ell(\beta)$.  By Lemmas \ref{perimeterSum}, \ref{ineq1}, and \ref{ineq2},
\begin{align*}
2(\ell(\alpha)-\ell(\alpha')) \leq \PS(M_H) &\leq \tfrac{1}{1-6\lambda}\ell(\bd M_H) = \tfrac{1}{1-6\lambda}\left(\ell(\alpha)+\ell(\beta)\right)\\
(1-12\lambda)\ell(\alpha) &\leq (2-12\lambda)\ell(\alpha')+\ell(\beta) < \ell(\alpha')+\ell(\beta) < 3\ell(\beta)\\
\ell(\alpha) &< \tfrac{3}{1-12\lambda}\ell(\beta).
\end{align*}
Therefore $u$ is $\left(\tfrac{3}{1-12\lambda}, 0\right)$-quasigeodesic, as desired.
\end{proof}

For this to be a meaningful bound we must have $0 < \lambda < \sfrac{1}{12}$, explaining our initial choice of $\lambda$.

\section{Proof of the main result}

In this section we prove the following proposition.

\begin{proposition}\label{mainProp}
Let $m,n \in \mathbb Z^+ \cup \{\infty\}$ with $m < n$.  Then there exist finitely generated, recursively presented groups $G$ and $B$ such that $B \leqslant G$ and
\begin{align*}
1 &\leq \asdim(G) \leq 2\\
m+1 &\leq \asdimAN(G) \leq m+2\\
n+1 &\leq \asdimAN(B) \leq n+2 \, .
\end{align*}
\end{proposition}

Since the proof requires many auxiliary lemmas, we again `globally fix' all notation in this section.  

Let $m$ be a fixed positive integer, and let $n \in \mathbb Z^+ \cup \{\infty\}$ with $m < n$.  Let $(\ell_i)$ be an increasing sequence of positive integers with $\ell_0 \geq 2$.  Let $S_A, S_B$ be disjoint finite sets, and let $0<\lambda<\sfrac{1}{12}$.  Suppose we have two languages
\begin{align*}
U_A = \{u_i \mid i \in \mathbb N\} \subset (S_A)_\circ^* && V_B = \{v_i \mid i \in \mathbb N\} \subset (S_B)_\circ^*
\end{align*}
satisfying all of the following conditions for all $i, i', j \in \mathbb N$.
\begin{enumerate}[label=\normalfont(\alph*)]
\item $U_A, V_B$ are cyclically minimal and cyclically reduced, and satisfy $C'(\lambda)$.
\item There exists a nonempty word $y \in (S_B)_\circ^*$ such that, for all $h \in \mathbb Z$, if $p$ is a piece of $y^h$ and $v_i$, then $|p| < \lambda |v_i|$.
\item $2 \leq |u_i| \leq |v_i|$.
\item If $u_i = u_{i'}$ or $v_i = v_{i'}$, then $i=i'$.
\item The sequence of word lengths $(|u_i|)$ is constant on blocks of the partition $\mathcal P_m$ and $(|v_i|)$ is constant on blocks of $\mathcal P_n$ (see \cref{partitions}).
\item $|u_{(j+1)m}| \geq \ell_{(j+1)m}|u_{jm}|$.  If $n \in \mathbb Z^+$ then $|v_{(j+1)n}| \geq \ell_{(j+1)n}|v_{jn}|$, and if $n=\infty$ then\\ $|v_{(j+1)^2}| \geq \ell_{(j+1)^2}|v_{j^2}|$.
\item $U_A, V_B$ are recursive.
\end{enumerate}

We construct an example of languages $U_A, V_B$ satisfying (a)-(f) in the next section, and show that they can be recursive in the process.  Assuming we already have $U_A, V_B$ satisfying (a)-(g), let $H_A, H_B$ be given by the presentations
\begin{align*}
H_A = \langle S_A \mid U_A \rangle && H_B = \langle S_B \mid V_B \rangle
\end{align*}
and let $A, B$ be central extensions of $H_A, H_B$, respectively, defined by
\begin{align*}
A &= \langle S_A \mid R_A \rangle := \langle S_A \mid [a,u_i], u_i^{\ell_i}: a \in S_A, i \in \mathbb N \rangle\\
B &= \langle S_B \mid R_B \rangle := \langle S_B \mid [b,v_i], v_i^{\ell_i}: b \in S_B, i \in \mathbb N \rangle \, .
\end{align*}
Since all elements in $R_A, R_B$ represent the trivial element in $H_A, H_B$, respectively, there are natural epimorphisms $\pi_A: A \to H_A$ and $\pi_B: B \to H_B$.  Recall that for a word $w$ in $(S_A)_\circ^*$ or $(S_B)_\circ^*$, we denote by $\bar w$ the element of $A$ or $B$, respectively, that $w$ represents.  Let 
\begin{align*}
K_A &= \Ker(\pi_A) = \langle \bar u_i : i \in \mathbb N \rangle \leqslant Z(A)\\
K_B &= \Ker(\pi_B) = \langle \bar v_i : i \in \mathbb N \rangle \leqslant Z(B)
\end{align*}
where we consider $K_A$ as a normed group, equipped with the restriction to $K_A$ of the word norm on $A$ with respect to the generating set $S_A$, which we will denote $\|\cdot\|_A$: similarly for $K_B$.

By condition (c), there exist sequences $s=(s_j), t=(t_j)$ such that $|u_i| = (m \times s)_i$ and $|v_i| = (n \times t)_i$ for all $i \in \mathbb N$.  Define normed groups $K_m$, $K_n$ similar to the normed group defined in \cref{K_nConstruction}, as follows:
\begin{align*}
K_m = \bigoplus_{i \in \mathbb N} |u_i| \mathbb Z_{\ell_i} = \bigoplus_{i \in \mathbb N} (m \times s)_i \mathbb Z_{\ell_i} && K_n = \bigoplus_{i \in \mathbb N} |v_i| \mathbb Z_{\ell_i} = \bigoplus_{i \in \mathbb N} (n \times t)_i \mathbb Z_{\ell_i} \, .
\end{align*}
Suppose that $x$ is a word over $S_A$ satisfying (b) with respect to $U_A$, except possibly the condition that $x$ not be the empty word.  Now condition (d) guarantees that $s$ and $t$ are increasing sequences of positive integers, such that for all $j \in \mathbb N$, $s_{j+1} \geq s_j\ell_{(j+1)m}$ and $t_{j+1} \geq t_j\ell_{(j+1)n}$ if $n \in \mathbb Z^+$.  Condition (e) guarantees that $s_0 \geq 2$ and $t_0 \geq 2$.  Therefore all hypotheses of Lemmas  \ref{ANLowerBound} and \ref{K_nConstruction} are satisfied, and we have
\begin{align*}
\asdimAN(|x|\mathbb Z \times K_m) =
\begin{cases}
m &\text{ if } x=\varepsilon\\
m+1 &\text{ otherwise}
\end{cases}
&& \asdimAN(|y|\mathbb Z \times K_n) = n+1 \, .
\end{align*} 
Now $K_A$ is abelian, $K_A$ satisfies $\bar u_i^{\ell_i}=1$ for all $i \in \mathbb N$, and, since $K_A$ is central in $A$, we have $\langle \bar x, K_A \rangle \cong \langle \bar x \rangle \times K_A$.  All the corresponding statements hold for $y$ and $K_B$.  Therefore there exist natural epimorphisms $\phi_A$ and $\phi_B$ defined by
\begin{align*}
\phi_A: |x|\mathbb Z \times K_m \to \langle \bar x, K_A \rangle && \phi_B: |y|\mathbb Z \times K_n \to \langle \bar y, K_B \rangle\\
(h, z) \mapsto \bar x^h \prod_{i \in \mathbb N} \bar u_i^{z_i} && (h, z) \mapsto \bar y^h \prod_{i \in \mathbb N} \bar v_i^{z_i}
\end{align*}
for all $h \in \mathbb Z$ and $z=(z_i) \in K_m$ or $K_n$.  In the case that $x = \varepsilon$ we have that $|x| = 0$ and $0 \mathbb Z = \{0\}$, so $\phi_A : K_m \to K_A$.

\begin{lemma}\label{biLipschitz}
Each of the epimorphisms $\phi_A, \phi_B$ is bi-Lipschitz, hence each is a quasi-isometry and an isomorphism.
\end{lemma}
\begin{proof}
We prove the statement for $\phi_A$.  Let $\|\cdot\|$ be the norm on $K_m$.  Let $h \in \mathbb Z$ and $z = (z_i) \in K_m$.  Let $(k_i)$ be the geodesic form of $z$ (see \cref{geodForm}).  Then
$$\|\phi_A(h, z)\|_A = \left\|\bar x^h \prod_{i \in \mathbb N} \bar u_i^{k_i}\right\|_A \leq \left| x^h \prod_{i \in \mathbb N} u_i^{k_i} \right| = h|x| + \sum_{i \in \mathbb N} |k_i||u_i| = \|(h, z)\|.$$
Now $k_i \leq \ell_i/2$ for all $i \in \mathbb N$, and $x^h$ satisfies condition (c) of \cref{quasiGeod}.  Furthermore,
$$A = \langle S_A \mid [a, u_i], u_i^{\ell_i} : a \in S_A, i \in \mathbb N \rangle = \langle S_A \mid [a, u_i], u_i^{\ell_i}, u_i(u_i)^{-1} : a \in S_A, i \in \mathbb N \rangle$$
and $U_A \cup U_A = U_A = \{u_i \mid i \in \mathbb N\}$ is a cyclically reduced, cyclically minimal $C'(\lambda)$ language, where $2 \leq |u_i| \leq |u_i|$ and $u_i = u_{i'}$ implies that $i = i'$ for all $i, i' \in \mathbb N$.  Thus we may apply \cref{quasiGeod} with $G = A, U = U_A, V = U_A,$ and $\tilde u = x^h$.  This yields
$$\|(h,z)\| = h|x| + \sum_{i \in \mathbb N}|u_i||k_i| = \left|x^h \prod_{i \in \mathbb N} u_i^{k_i} \right| \leq \left(\frac{3}{1-12\lambda}\right)\left\|\bar x^h \prod_{i \in \mathbb N} \bar u_i^{k_i}\right\|_A = \left(\tfrac{3}{1-12\lambda}\right)\|\phi_A(h,z)\|_A \, ,$$
hence $\left(\frac{1-12\lambda}{3} \right) \|(h,z)\| \leq \|\phi_A(k,z)\|_A \leq  \|(h,z)\|$ and $\phi_A$ is bi-Lipschitz.
\end{proof}

By replacing $x$ or $y$ with $\varepsilon$, we obtain the following.

\begin{corollary}\label{restrictedbiLipschitz}
Both $\phi_A|_{K_m} : K_m \to K_A$ and $\phi_B|_{K_n} : K_n \to K_B$ are bi-Lipschitz maps. Therefore $\asdimAN(K_A) = \asdimAN(K_m) = m$ and $\asdimAN(K_B) = \asdimAN(K_n) = n$.
\end{corollary}

In order to get our bounds on $\asdimAN(G)$ and $\asdimAN(B)$, we use the extension theorems for asymptotic and Assouad-Nagata dimension.

\begin{lemma}[Extension Theorems]\cite{Bell_Dranishnikov2, Brodskiy_etal}\label{extensionThm}
Let
$$1 \to K \to G \to H \to 1$$
be a short exact sequence, where $G$ and $H$ are finitely generated groups equipped with the word norm with respect to some finite generating set, and the norm on $K$ is the restriction to $K$ of the norm on $G$.  Then 
\begin{align*}
\asdim(G) &\leq \asdim(K)+\asdim(H)\\
\asdimAN(G) &\leq \asdimAN(K) + \asdimAN(H) \, .
\end{align*}
\end{lemma}

\begin{lemma}\cite{Sledd}\label{C'(1/6)Lemma}
Let $H$ be a finitely generated $C'(\sfrac{1}{6})$ group.  Then $\asdimAN(H) \leq 2$.
\end{lemma}

\begin{corollary}\label{ABBounds}
We have
\begin{align*}
1 \leq \asdim(A) &\leq 2 && 1 \leq \asdim(B) \leq 2\\
m \leq \asdimAN(A) &\leq m+2 && n+1 \leq \asdimAN(B) \leq n+2 \, .
\end{align*}
Also, if $x \neq \varepsilon$, then $\asdimAN(A) \geq m+1$.
\end{corollary}
\begin{proof}
We establish the bounds for $A$: the argument for $B$ is similar.  Since $A$ is finitely generated and infinite, $\asdimAN(A) \geq 1$.  By \cref{restrictedbiLipschitz}, $\asdimAN(A) \geq \asdimAN(K_A) = m$.  If $x \neq \varepsilon$,
$$\asdimAN(A) \geq \asdimAN(\langle \bar x, K_A \rangle) = \asdimAN(|x|\mathbb Z \times K_m) = m+1$$
since $|x| > 0$.  This gives the lower bounds on the asymptotic and Assouad-Nagata dimension of $A$.  For the upper bounds, note that $A$ is constructed so that there is a short exact sequence
$$1 \to K_A \to A \to H_A \to 1$$
where $H_A$ is a finitely generated $C'(\sfrac{1}{6})$ group and hence $\asdim(H_A) \leq \asdimAN(H_A) \leq 2$.  Since $K_A$ is locally finite, $\asdim(K_A)= 0$.  Now by \cref{extensionThm},
\begin{align*}
\asdim(A) &\leq \asdim(K_A)+\asdim(H_A) \leq 2\\
\asdimAN(A) &\leq \asdimAN(K_A)+\asdimAN(H_A) \leq m+2 \, .
\end{align*}
\end{proof}
By \cref{restrictedbiLipschitz}, the maps $\phi_A|_{K_m} : K_m \to K_A$ and $\phi_B|_{K_n}: K_n \to K_B$ are isomorphisms.  Therefore the map defined by $\bar u_i \mapsto \bar v_i$ for all $i \in \mathbb N$ extends to an isomorphism from $K_A$ to $K_B$.  Let $\phi: K_A \to K_B$ be this isomorphism. Let
$$G = A *_\phi B := \langle A \sqcup B \mid a\phi(a)^{-1}: a \in A \rangle \, .$$
Let $S = S_A \sqcup S_B$.  Then $G$ admits the presentation
$$G = \langle S \mid R_G \rangle := \langle S \mid [s, u_i], u_i^{\ell_i}, u_iv_i^{-1} : s \in S, i \in \mathbb N \rangle \, ,$$
which is recursive if $U_A$ and $V_B$ are recursive.  Let 
$$H = \langle S \mid R_H \rangle := \langle S_A \sqcup S_B \mid U_A \sqcup V_B \rangle = \langle S_A \sqcup S_B \mid u_i, v_i: i \in \mathbb N \rangle \, .$$
Since $U_A$ and $V_B$ are $C'(\lambda)$ languages over disjoint alphabets, $H$ is a $C'(\lambda)$ group.  Furthermore, all words in $R_G$ represent the trivial element of $H$, so there is a natural epimorphism $\pi : G \to H$.  Let $K = \Ker(\pi)$.  Then
$$K = \langle \bar u_i : i \in \mathbb N \rangle \leqslant Z(G) \, .$$
We consider $K$ as a normed group, where the norm on $K$ is the restriction to $K$ of the word norm on $G$ with respect to $S$.  Thus we have a short exact sequence
$$1 \to K \to G \to H \to 1 \, .$$
Let $b \in S_B$.  Considering the relations of $K$ and the fact that $K$ is central in $G$, there exists a natural epimorphism $\phi_K: \mathbb Z \times K_m \to \langle \bar b, K \rangle$ given by
\begin{align*}
\phi_K (h,z) = \bar b^h \prod_{i \in \mathbb N} \bar u_i^{z_i}
\end{align*}
for all $h \in \mathbb Z$ and $z = (z_i) \in K_m$.  Now we have the following.

\begin{lemma}\label{GbiLipschitz}
The epimorphism $\phi_K$ is bi-Lipschitz, in particular $\phi_K$ is a quasi-isometry and an isomorphism.
\end{lemma}
\begin{proof}
The proof is similar to that of \cref{biLipschitz}.  The only difference is that now we apply \cref{quasiGeod} with $U=U_A, V=V_B,$ and $\tilde u = b^h$.  Since $b$ is a word over an alphabet disjoint from $S_A$, clearly condition (c) of \cref{quasiGeod} is satisfied with $\tilde u = b^h$ for any $h \in \mathbb N$.  Since $2 \leq |u_i| \leq |v_i|$ and $u_i \neq v_{i'}$ for all $i, i' \in \mathbb N$, all hypotheses of \cref{quasiGeod} are satisfied.
\end{proof}

We are now ready to prove \cref{mainProp}.

\begin{proof}[Proof of \cref{mainProp}]
Let $B,G$ be defined as in this section.  The bounds on $\asdimAN(B)$ are established in \cref{ABBounds}.  Since $G$ is finitely generated and infinite, $\asdim(G) \geq 1$.  For the lower bound on the Assouad-Nagata dimension of $G$, note that
$$\asdimAN(G) \geq \asdimAN(\langle \bar b, K \rangle) = \asdimAN(\mathbb Z \times K_m) = m+1 \, .$$
By \cref{C'(1/6)Lemma}, we have $\asdim(H) \leq \asdimAN(H) \leq 2$.  Applying the extension theorems to the short exact sequence $1 \to K \to G \to H \to 1$ yields that $\asdim(G) \leq 2$ and $\asdimAN(G) \leq m+2$.
\end{proof}

We give a presentation of a group $G$ satisfying the conditions of \cref{mainProp} in the next section.  For now, we derive the main result of this paper as a corollary.  To do this, we will need to recall two theorems of asymptotic dimension theory.  The first a theorem of Dranishnikov and Smith, known as the Morita theorem for asymptotic Assouad-Nagata dimension.  We state a special case of it here.

\begin{theorem}[Morita theorem for $\asdimAN$]\cite{Dranishnikov_Smith}
Let $G$ be a finitely generated group.  Then $\asdimAN(G \times \mathbb Z) = \asdimAN(G)+1$.
\end{theorem}

The second is the free product formulas for asymptotic and Assouad-Nagata dimension.  The theorem for $\asdim$ is due to Dranishnikov, and its counterpart for $\asdimAN$ is due to Brodskiy and Higes.

\begin{theorem}\cite{Dranishnikov, Brodskiy_Higes}
Let $A$ and $B$ be finitely generated groups.  Then
\begin{align*}
\asdim(A*B) &= \max\{\asdim(A), \asdim(B), 1\}\\
\asdimAN(A*B) &= \max \{\asdimAN(A), \asdimAN(B), 1\}
\end{align*}
\end{theorem}

We are now ready to prove the main theorem.

\begin{theorem}
For all $k,m,n \in \mathbb N \cup \{\infty\}$ with $4 \leq k \leq m \leq n$, there exist finitely generated, recursively presented groups $G$ and $H$ with $H \leqslant G$, such that
\begin{align*}
\asdim(G) &= k\\
\asdimAN(G) &= m\\
\asdimAN(H) &= n \, .
\end{align*}
\end{theorem}
\begin{proof}
Applying \cref{mainProp} with $m-3$ and $n-2$, there exist finitely generated, recursively presented groups $G_0$ and $B_0$ with $B_0 \leqslant G_0$, such that
\begin{align*}
1 &\leq \asdim(G_0) \leq 2\\
m-2 &\leq \asdimAN(G_0) \leq  m-1\\
n-1 &\leq \asdimAN(B_0) \leq n \, .
\end{align*}

Let
$$G_1 =
\begin{cases}
G_0 \times \mathbb Z^2 &\text{ if } \asdimAN(G_0) = m-2\\
G_0 \times \mathbb Z &\text{ if } \asdimAN(G_0) = m-1 \, .
\end{cases}$$
Then by the Morita theorem for Assouad-Nagata dimension, we have $\asdim(G_1) = m$.  By the extension theorem for asymptotic dimension, we have that $\asdim(G_1) \leq \asdim(G_0)+2 \leq 4$.  Now let $G = G_1 * \mathbb Z^k$.  Then since $4 \leq k \leq m$, by the free product formulas for asymptotic and Assouad-Nagata dimension it follows that $\asdim(G) = k$ and $\asdimAN(G) = m$.  Note that $B_0$ and $B_0 \times \mathbb Z$ are both subgroups of $G$.  Therefore, let
$$H =
\begin{cases}
B_0 \times \mathbb Z &\text{ if } \asdimAN(B_0) = n-1\\
B_0 &\text{ if } \asdimAN(B_0) = n \, .
\end{cases}$$
Again by the Morita theorem, we have that $\asdimAN(H) = n$, and $H \leqslant G$.  This completes the proof.
\end{proof}

\section{A concrete example}

In this section we construct an example of a group of the sort described in \cref{mainProp}.  In doing so, we show that such a group can be given by an explicit presentation, i.e. is recursively presented.  The following lemma shows one way of constructing $C'(\lambda)$ languages, which was used by Bowditch in \cite{Bowditch2} to construct $2^{\aleph_0}$ small cancellation groups in distinct quasi-isometry classes.

\begin{lemma}\label{LConstruction}
Let $U = \{u_i \mid i \in \mathbb N\} \subset \{a,x\}_\circ^*$ be a language where we define
$$u_i = (a^{m_i}x^{m_i})^{n_i}$$
for some positive integers $m_i, n_i$, for each $i \in \mathbb N$.  Let $k \geq 2$ be an integer, and suppose that all of the following conditions hold.
\begin{enumerate}[label=\normalfont(\alph*)]
\item $n_i \geq k$ for all $i \in \mathbb N$.
\item $m_i \neq m_{i'}$ for all distinct $i, i' \in \mathbb N$.
\end{enumerate}
Then all of the following conclusions hold for all $i \in \mathbb N$.
\begin{enumerate}[label=\normalfont(\roman*)]
\item $U$ is cyclically minimal and cyclically reduced, and satisfies $C'\left(\frac{1}{k-1}\right)$.
\item For all $h \in \mathbb Z$, if $p$ is a piece of $x^h$ and $u_i$, then $|p| < \frac{1}{k-1}|u_i|$.
\item $2 \leq |u_i|$.
\item If $u_i = u_{i'}$ then $i=i'$.
\end{enumerate}
\end{lemma}
\begin{proof}
If $u_i \in U$, then no cyclic shift of $u_i^{-1}$ is in $U$: if $\tilde u_i$ is a cyclic shift of $u_i$ that belongs to $U$, then $|\tilde u_i| = |u_i|$ and $\tilde u_i$ must begin with $a$ and end with $x$, in which case $\tilde u_i = u_i$.  Therefore $U$ is cyclically minimal.  Since all $u_i$ are positive words (that is, do not contain letters $a^{-1}$ or $x^{-1}$), it is clear that $U$ is cyclically reduced.  For the same reason, when talking about pieces of some $u_i$ and another positive word $w$, it suffices to consider only cyclic shifts of $u_i$ and $w$, and we may ignore cyclic shifts of $u_i^{-1}$ or $w^{-1}$.  To show that $U$ satisfies $C'(\frac{1}{k-1})$, suppose $i, i' \in \mathbb N$ are distinct.  Let $p$ be a maximal piece of $u_i$ and $u_{i'}$.  Since $m_i \neq m_{i'}$, suppose without loss of generality that $m_i < m_{i'}$.  Then $p$ must have the form $a^{m_i}x^{m_i}$.  But then $n_i|p| \leq |u_i|$ and $n_{i'}|p| \leq |u_{i'}|$.  Since $n_i, n_{i'} \geq k$, we have $|p| \leq \frac{1}{k}\min(|u_i|, |u_{i'}|) < \frac{1}{k-1}\min(|u_i|, |u_{i'}|)$.  Therefore $U$ satisfies $C'\left(\frac{1}{k-1}\right)$.  Conclusion (ii) says only that any power of $x$ makes up less than $\frac{1}{k-1}$ of a cyclic shift of some $u_i$.  But a maximal subword of a cyclic shift of $u_i$ of the form $x^h$ must be $x^{m_i}$, which has length at most $\frac{1}{2n_i}|u_i| \leq \frac{1}{2k}|u_i| < \frac{1}{k-1}|u_i|$, so this is clear.  Parts (iii) and (iv) are obvious.
\end{proof}

\begin{lemma}\label{technicalLanguageConstruction}
Let $m \in \mathbb Z^+ \cup \{\infty\}$, and let $\mathcal P_m = \{P_{(m,j)} \mid j \in \mathbb N\}$ be the partition of $\mathbb N$ given in \cref{partitions}.  Let $k \geq 2$ be an integer.  For each $i \in \mathbb N$, let $r_i = i - \min(P_{(m,j)})$ whenever $i \in P_{(m,j)}$.  Let $(p_j)$ be an increasing sequence of positive integers.   Let $U = \{u_i \mid i \in \mathbb N\} \subset \{a,x\}_\circ^*$ be given by
$$u_i = \left(a^{k^{(p_j-r_i)}}x^{k^{(p_j-r_i)}}\right)^{k^{(r_i+1)}}$$
whenever $i \in P_{(m,j)}$.  Let $(\ell_i)$ be an increasing sequence of positive integers.  Suppose that the sequence $(p_j)$ satisfies
\begin{equation}\label{p_jcondition}
\begin{split}
p_{j+1} \geq p_j+\log_k(\ell_{(j+1)m})+|P_{(m,j+1)}| \text{ if } m \in \mathbb Z^+\\
 p_{j+1} \geq p_j+\log_k(\ell_{(j+1)^2})+|P_{(m,j+1)}| \text{ if } m = \infty \, .
\end{split}
\end{equation}
Then all of the following conclusions hold for all $i \in \mathbb N$.
\begin{enumerate}[label=\normalfont(\roman*)]
\item $U$ is cyclically minimal and cyclically reduced, and satisfies $C'\left(\frac{1}{k-1}\right)$.
\item For all $h \in \mathbb N$, if $p \in \{a,x\}_\circ^*$ is a piece of $x^h$ and $w_i$, then $|p| < \frac{1}{k-1}|u_i|$.
\item $2 \leq |u_i|$.
\item If $u_i = u_{i'}$, then $i = i'$.
\item The sequence of word lengths $(|u_i|)$ is constant on blocks of $\mathcal P_m$.
\item If $m \in \mathbb Z^+$ then $|u_{(j+1)m}| \geq \ell_{(j+1)m}|u_{jm}|$, and if $m = \infty$ then $|u_{(j+1)^2}| \geq \ell_{(j+1)^2}|u_{j^2}|$.

\end{enumerate}
\end{lemma}
\begin{proof}
Note that, if $i \in P_{(m,j)}$, then $|u_i| = 2 k^{p_j-r_i}k^{r_i+1} = 2k^{p_j+1} \geq 2$, which depends only on $j$.  This establishes (iv) and (v).  Define the sequence $s=(s_j)$ by
$$s_j = 2k^{p_j+1}$$
for all $j \in \mathbb N$.  Then $|u_i| = (m \times s)_i$.  

For (vi), note that $\log_k(s_{(j+1)m}) = \log_k(2) + p_{j+1}+1$.  If $m \in \mathbb Z^+$, then we have $p_{j+1} \geq p_j+\log_k(\ell_{(j+1)m})$, implying that $s_{j+1} \geq \ell_{(j+1)m}s_j$ for all $j \in \mathbb N$.  If $m = \infty$, then $\log_k(s_{j+1}) \geq p_{j+1} \geq p_j+\log_k(\ell_{(j+1)^2})$, so $s_{j+1} \geq \ell_{(j+1)^2}s_j$ for all $j \in \mathbb N$.  This establishes (vi).

For parts (i)-(iv), we use \cref{LConstruction}.  Obviously part (a) of \cref{LConstruction} is satisfied, so we only need to check part (b).  For this is suffices to show that if $i \in P_{(m,j)}, i' \in P_{(m,j')}$, and $i \neq i'$, then  $p_j-r_i \neq p_{j'}-r_{i'}$.  If $j'=j$ then this is immediate.  If $j' = j+1$ then we have
$$p_{j'}-r_{i'} = p_{j+1}-r_{i'} \geq p_{j+1}-|P_{j+1}|+1 > p_j \geq p_j-r_i \, .$$
This shows that $p_j-r_i$ increases with $j$ no matter the choice of $i \in P_{(m,j)}$, so we are done.
\end{proof}

We are ready to construct our example.  Let $m, n \in \mathbb Z^+ \cup \{\infty\}$ with $m < n$.  Let $S_A = \{a,x\}$, $S_B = \{b,y\}$ be disjoint two-element alphabets.  Let $k=14$ and let $\ell_i = 14^i$ for all $i \in \mathbb N$.  Let $(p_j), (q_j)$ be increasing sequences of positive integers.  Let $U_A = \{u_i \mid i \in \mathbb N\} \subset (S_A)_\circ^*$ be the language constructed with respect to $m,k, (\ell_i)$ and $(p_j)$ as in \cref{technicalLanguageConstruction}.  Similarly define $V_B = \{v_i \mid i \in \mathbb N\} \subset (S_B)_\circ^*$ with respect to $n,k, (\ell_i)$, and $(q_j)$.

\begin{lemma}\label{pqSequenceConstruction}
Suppose that for all $i,j \in \mathbb N$ we have
\begin{enumerate}[label=\normalfont(\alph*)]
\item $p_{j+1} \geq p_j + (j+2)m$.
\item $q_{j+1} \geq q_j + (j+2)n$ if $n \in \mathbb Z^+$, and $q_{j+1} \geq q_j + (j+2)^2$ if $n=\infty$.
\item $p_{\lfloor i/m \rfloor} \leq q_{\lfloor i/n \rfloor}$ if $n \in \mathbb Z^+$, and $p_{\lfloor i/m \rfloor} \leq q_{\lfloor \sqrt{i} \rfloor}$ if $n=\infty$.
\end{enumerate}
Then $U_A, V_B$ satisfy conditions \textnormal{(a)-(f)} listed in the proof of \cref{mainProp}.
\end{lemma}
\begin{proof}
Note that
\begin{align*}
\log_k(\ell_{(j+1)n})+|P_{(n,j+1)}| &= \log_{14}(14^{(j+1)n})+n = (j+2)n &&\text{ if } n \in \mathbb Z^+\\
\log_k(\ell_{(j+1)^2})+|P_{(n,j+1)}| &= \log_{14}(14^{(j+1)^2})+ (2j+1) \leq (j+2)^2 &&\text{ if } n = \infty \, .
\end{align*}
Therefore assumptions (a) and (b) guarantee that $(p_j)$ and $(q_j)$ satisfy (\ref{p_jcondition}) with respect to $(\ell_i)$ and $m,n$, respectively, and so $U_A, V_B$ satisfy all conditions listed in the proof of \cref{mainProp}, except possibly that $|u_i| \leq |v_i|$ for all $i \in \mathbb N$.  Now, if $i \in P_{(n,j)} \in \mathcal P_n$, then $j = \lfloor i/n \rfloor$ if $n \in \mathbb Z^+$, and $j = \lfloor \sqrt{i} \rfloor$ if $n = \infty$.  It follows that assumption (c) is necessary and sufficient to guarantee that $|u_i| \leq |v_i|$ for all $i \in \mathbb N$.
\end{proof}

\begin{example}
Let
\begin{align*}
p_j &= m(j+2)^2
&& q_j =
\begin{cases}
n^2(j+3)^2 &\text{ if }n \in \mathbb Z^+\\
m(j+3)^4 &\text{ if } n = \infty.
\end{cases}
\end{align*}
Then $(p_j), (q_j)$ satisfy the hypotheses of  \cref{pqSequenceConstruction}.  The verification of this is no more than a tedious calculation, so we omit it.  Note that, in the notation of \cref{technicalLanguageConstruction},
$$r_i = 
\begin{cases}
i \mod n &\text{ if } n \in \mathbb Z^+\\
i^2-\lfloor \sqrt i \rfloor^2 &\text{ if } n = \infty \, .
\end{cases}$$
Also, if $i \in P_{(n,j)}$, then $j = \lfloor i / n \rfloor$ if $n \in \mathbb Z^+$, and $j = \lfloor \sqrt i \rfloor$ if $n = \infty$.  So, expanding the forms of $u_i$ and $v_i$ according to \cref{technicalLanguageConstruction} with respect to the sequences $(p_j)$ and $(q_j)$ given above yields

\begin{align*}
u_i &= \left(a^{14^{m(\lfloor i/m \rfloor+2)^2-(i\text{ mod } m)}}x^{14^{m(\lfloor i/m \rfloor+2)^2-(i\text{ mod } m)}}\right)^{14^{(i \text{ mod } m) +1}}\\
v_i &= 
\begin{cases}
\left(b^{14^{n^2(\lfloor i/n \rfloor+3)^2-(i \text{ mod } n)}}y^{14^{n^2(\lfloor i/n \rfloor+3)^2-(i \text{ mod } n)}}\right)^{14^{(i \text{ mod } n)+1}} & \text{ if } n \in \mathbb Z^+\\
\left(b^{14^{m(\lfloor \sqrt i \rfloor+3)^4-(i-\lfloor \sqrt i \rfloor^2)}}y^{14^{m(\lfloor \sqrt i \rfloor+3)^4-(i-\lfloor \sqrt i \rfloor^2)}}\right)^{14^{(i-\lfloor \sqrt i \rfloor^2)+1}} &\text{ if } n = \infty \, .
\end{cases} 
\end{align*}
Then the languages $\{u_i \mid i \in \mathbb N\}$ and $\{v_i \mid i \in \mathbb N\}$ satisfy conditions (a)-(f) listed in the proof of \cref{mainProp}, and are clearly recursive.  Thus the group $G$ with presentation
$$G = \langle a, b, x, y \mid [a,u_i], [x,u_i], [b,u_i], [y,u_i], u_i^{14^i}, u_iv_i^{-1} : i \in \mathbb N \rangle$$
is a finitely generated, recursively presented group of Assouad-Nagata dimension at most $m+2$, containing a finitely generated subgroup of Assouad-Nagata dimension at least $n+1$.

\end{example}

\pagebreak

\end{document}